\documentclass[12pt,reqno]{amsart}
\usepackage{amsmath,amsfonts,graphicx,amscd,amssymb,epsf,amsthm,enumerate,parskip,mathrsfs,color,ltablex,blindtext,multirow}  
 \usepackage[a4paper, margin=1in]{geometry}
\usepackage{times}
  \usepackage{adjustbox}
\theoremstyle{definition}
\newtheorem{theorem}{Theorem}[section]

\newtheorem{lemma}[theorem]{Lemma}
\newtheorem{corollary}[theorem]{Corollary}
\newtheorem{definition}{Definition}[section]
\newtheorem{remark}[definition]{Remark}
 \numberwithin{equation}{section}
\numberwithin{equation}{section}
 
\newcommand{\di}{\displaystyle}
\newcommand{\mb}{\mathbb}
\newcommand{\hsp}{\hspace{0.3cm}}
\usepackage{hyperref}
\hypersetup{
	colorlinks=true,
	linkcolor=blue,
	citecolor=blue,
}

\begin{document}
  
\title[Upper Bound of Generalized Divisor Functions]{Superior Highly Composite Numbers and the Explicit Upper Bound of Generalized Divisor Functions}
\author{Lee-Peng Teo}
 \address{Department of Mathematics, Xiamen University Malaysia\\Jalan Sunsuria, Bandar Sunsuria, 43900, Sepang, Selangor, Malaysia.}
\email{lpteo@xmu.edu.my}
\begin{abstract}
For $k\geq 2$, we give a detailed exposition of the superior $k$-highly composite numbers. We then consider the function
\[f_k(n)=\frac{\log d_k(n)\log\log n}{\log k\log n},\quad n\geq 3\]
which has a maximum value $\lambda(k)$ at a superior $k$-highly composite number. We develop an efficient algorithm to compute $\lambda(k)$ and the positive integer $N_{\max}(k)$ where $f_k$ achieves the value $\lambda(k)$. The results for $2\leq k\leq 100$ are tabled.
 
\end{abstract}
\maketitle

\tableofcontents

\vfill\pagebreak
\listoftables

\vfill\pagebreak
\section{Introduction}

Given   positive integers $k$ and $n$, let $d_k(n)$ be the number of ways to write $n$ as a product of $k$ positive integers, with order taken into account. 
When $k=2$, we write $d_2(n)$ simply as $d(n)$. It is the number of positive divisors of  $n$. When $k\geq 3$, we have the relation
\[d_{k}(n)=\sum_{q|n}d_{k-1}(q).\]
The function $d_k(n)$ is multiplicative. For a prime $p$ and a positive integer $m$,
\[d_k(p^m)=\binom{k+m-1}{m}=\frac{k(k+1)\cdots (k+m-1)}{m!}.\]
It is a classical result  that \begin{equation}\label{20250731_1}\log d_k(n)=O\left(\frac{\log n}{\log\log n}\right).\end{equation}
When $k=2$, this is proved in \cite{Wigert}, cf.\ also \cite{Nicolas}. For $k>2$, \eqref{20250731_1} follows easily from Lemma \ref{20250723_1} below.

Define the function $f_k:\{n\in\mb{Z}^+\,|\,n \geq 3\}\to \mb{R}$ by
\begin{equation}\label{20250723_2}
f_k(n)=\frac{\log d_k(n) \log\log n}{\log k \log n}.\end{equation}Then $f_k(n)\geq 0$ for all $n\geq 3$. Since $d_k(n)=k$ whenever $n=p$ is a prime, we find that 
\[\liminf_{n\to\infty}f_k(n)=0.\]
On the other hand,
Eq. \eqref{20250731_1} implies that the function  $f_k(n)$  is bounded above. Therefore,
$\di \limsup_{n\to\infty}f_k(n)$ exist. 
In fact,  $f_k(n)$ has a maximum value. Namely,  
\begin{equation}\label{20250805_1}\lambda(k)=\max_{n\geq 3}f_k(n)\end{equation} exists, and  there is  a positive integer $N_{\max}(k)$ such that
\[\lambda(k)= f_k(N_{\max}(k))\geq f_k(n)\hspace{1cm}\text{for all}\;\;n\geq 3.\]

In \cite{Nicolas}, Nicolas and Robin used the concept of superior highly composite numbers introduced by Ramanujan \cite{Ramanujan} to show that  $\lambda(2)=1.5379 $ and \[N_{\max}(2)=2^5\times 3^3\times 5^2\times 7\times 9\times 11\times 13\times 17\times 19=6 983 776 800.\] In other words,
\[\log d(n)\leq\frac{1.5379 \log 2\log n}{\log\log n}\hspace{1cm}\text{for all}\;n \geq 3,\]
 and the equality is achieved when $n=6 983 776 800$.  For $k=3$,  the authors of \cite{Duras} referred to \cite{Robin} for the result $\lambda(3)=1.5914$\footnote{But this result does not seem to appear in \cite{Robin}.}, which gives
\[\log d_3(n)\leq \frac{1.5914\log 3\log n}{\log\log n}\hspace{1cm}\text{for all}\;n \geq 3.\]
The behaviors of the values of $\lambda(k)$ were studied in \cite{Duras}, especially when $k$ is large. 
To  the best of our knowledge, the precise values of $\lambda(k)$ for $k\geq 4$ have not been computed explicitly, although it was claimed in \cite{Duras} that they would appear  in a future publication. In principle, one can compute $\lambda(k)$ for any $k\geq 2$ by mimicking the algorithm for $\lambda(2)$ given in \cite{Nicolas}. But there are some auxiliary numbers used in \cite{Nicolas} whose generalization to  any $k\geq 3$ are not clear, although the authors of \cite{Duras} have attempted to tackle some of these problems.

In this work,   we resolve this problem.
As in \cite{Duras}, for any $\varepsilon>0$,   a positive integer $N_{k,\varepsilon}$ is a superior $k$-highly composite number   associated to $\varepsilon$ if 
\[\frac{d_k(n)}{n^{\varepsilon}}\leq\frac{d_k(N_{k,\varepsilon})}{N_{k,\varepsilon}^{\varepsilon}}\hspace{1cm}\text{for all}\;\;n\in\mb{Z}^+.\]
In \cite{ Duras}, it was shown that $N_{\max}(k)$, the maximizer of $f_k(n)$, must be a superior $k$-highly composite number   associated to some positive number $\varepsilon(k)$. They proved that
\[\lim_{k\to\infty}\varepsilon(k)=\frac{\log k}{\log 4},\]and
\[\lambda(k)=\frac{\log k}{\log 16}+\frac{\log\log 2+1}{\log 4}+O\left(\frac{1}{\log k}\right)\hspace{1cm}\text{when}\;\;k\to\infty.\]   In this work, we are interested in explicit values of $\lambda(k)$ and $N_{\max}(k)$ for all $k\geq 2$. For this, we justify rigorously some effectively computable bounds of $\varepsilon(k)$ for any $k\geq 2$. This will leave us with finitely many candidates for $N_{\max}(k)$. From there, we can determine $N_{\max}(k)$ and   $\lambda(k)$ explicitly.  Explicit values of  $\lambda(k)$ and $N_{\max}(k)$   for $2\leq k\leq 100$ are tabled.

	\vspace{0.5cm}
	\noindent
	{\bf Acknowledgment.}   We are grateful to the reviewer for giving us  valuable comments.

\bigskip
 \section{Preliminaries}\label{bounds}
 In this section, we present some results that we need in subsequent sections. 
 We define $d_1(n)=u(n)$ to be the unit function that is equal to 1 identically. For $k\geq 2$, the function $d_k(n)$ is defined recursively by
 \[d_k(n)=(d_{k-1}*u)(n)=\sum_{q\mid n}d_{k-1}(q).\]
It is a multiplicative function with Dirichlet series
 \[\sum_{n=1}^{\infty}\frac{d_k(n)}{n^s}=\zeta^k(s)=\prod_{p}\left(1-\frac{1}{p^s}\right)^{-k}.\]
Since
 \[(1-z)^{-k}=\frac{1}{(k-1)!}\sum_{m=0}^{\infty} (m+k-1)\cdots (m+1) z^m=\sum_{m=0}^{\infty}\binom{m+k-1}{k-1}z^m,\]
 we find that if $n=p_1^{m_1}\cdots p_r^{m_r}$ is the prime factorization of $n$, then
 \begin{equation}\label{dkn}d_k(n)=\prod_{i=1}^r \binom{m_i+k-1}{k-1}.\end{equation}
 In particular, when $p$ is prime, $d_k(p)=k$. Therefore, $d_k(n)=k$ infinitely often.
 Let $f_k(n)$ be the function defined as \[f_k(n)=\frac{\log d_k(n)\log\log n}{\log k\log n}\hspace{1cm}\text{for}\;\;n\geq 3.\] Then $f_k(n)\geq 0$ for all $n\geq 3$. For $r\geq 1$, let $p_r$ be the $r$-th prime number. Then $\di\lim_{r\to\infty}p_r=\infty$. Therefore,
\[\lim_{r\to\infty}f_k(p_r)=\lim_{k\to\infty}\frac{\log\log p_r}{\log p_r}=0.\]This show that
\[\liminf_{n\to 0}f_k(n)=0.\]

 It is obvious that if $q$ and $n$ are positive integers and $q\mid n$, then $d_k(q)\leq d_k(n)$.
 The following lemma gives a simple inequality comparing $d_k(n)$ with $d(n)$. 
 \begin{lemma}\label{20250723_1}
 For $k\geq 2$,
 \begin{equation}\label{20250802_3}d_k(n)\leq d(n)^{k-1}.\end{equation}
 \end{lemma}
 \begin{proof}
 We prove the statement by induction on $k$. For $k=2$, it is a tautology. Assume that $k\geq 3$ and 
 \[d_{k-1}(n)\leq d(n)^{k-2}.\]
 Then since
 \[d_k(n)=\sum_{q\mid n}d_{k-1}(q),\]
 and for each $q\mid n$, $d_{k-1}(q)\leq d_{k-1}(n)$, we find that
 \[d_k(n)\leq d_{k-1}(n)\sum_{q\mid n}1=d_{k-1}(n)d(n).\]
 It follows from the induction hypothesis that
 \[d_k(n)\leq d(n)^{k-1},\]which completes the proof by induction. \qedhere
 \end{proof}
 When $p$ is a prime and $m$ is a nonnegative integer, \eqref{dkn} gives
\begin{equation}\label{dkn2}d_k(p^m)=\binom{m+k-1}{k-1}=\binom{m+k-1}{m}=\frac{k(k+1)\cdots (k+m-1)}{m!}.\end{equation}Using this, one can show that  \begin{equation}\label{dkpm}d_k(p^m)\leq k^m.\end{equation}
Indeed, if $m=0$ then $d_k(1)=1=k^0$. If $m\geq 1$, we have from \eqref{dkn2}
 \[d_k(p^m)=\prod_{i=1}^m\frac{k+i-1}{i}.\]
For any $i\geq 1$,
\[ki-(k+i-1)=(k-1)(i-1)\geq 0.\]
It follows that
\[d_k(p^m)\leq\prod_{i=1}^m \frac{ki}{i}=k^m.\]
Let
\[\pi(x)=\sum_{p\leq x}1\] be the prime counting function, and let
\[\vartheta(x)=\sum_{p\leq x}\log p\hspace{1cm}\text{and}\hspace{1cm}\psi(x)=\sum_{p^m\leq x}\log p \] be  respectively the Chebyshev's theta function and psi function. Then $\vartheta(x)=\psi(x)=0$ when $x<2$, and
\[\psi(x)=\sum_{m=1}^{\infty}\vartheta(x^{1/m}).\]

We will need  the following results later.
\begin{lemma}\label{20250723_3}
We have the following bounds.
\begin{enumerate}[(i)]
\item For $x>1$, $\pi(x)\leq  \di\frac{  \omega_1x}{ \log x}$, where $\omega_1=1.2551$.
\item For $x\geq 8$, $\pi(x)\leq \di\frac{x}{2}$. 
\item For $x\geq 0$, $\vartheta(x)\leq \omega_2 x$, where $\omega_2=1.00001$.

\end{enumerate}
\end{lemma}
\begin{proof}
Statement (i) follows from Corollary 1 of \cite{RosserSchoenfeld1962}. Statement (ii) follows easily from statement (i) or more simply from the fact that half of the integers are even. Statement (iii) follows from  Corollary 2.1 of \cite{Broadbent} where
$\vartheta(x)\leq \left(1+1.93378\times 10^{-8}\right)x$ for all $ x\geq 0$ was proved.
\end{proof}

\bigskip
 \section{Superior Highly Composite Numbers}
From now on, we will concentrate on a fixed $k\geq 2$. 
The concepts of highly composite and  superior  highly composite numbers were introduced by Ramanujan \cite{Ramanujan}. In the same paper, he has also defined $k$-highly  composite numbers (see \cite{Ramanujan_2}), so that the highly composite numbers introduced earlier  are $2$-highly composite. In \cite{Duras}, superior $k$-highly composite numbers were introduced, generalizing the superior  highly composite numbers ($k=2$ case) defined by Ramanujan \cite{Ramanujan}. 
\begin{definition}

A positive integer $N$ is said to be $k$-highly composite   if 
\[d_k(n)<d_k(N)\hspace{1cm}\text{for all}\; 1\leq n< N.\]
A positive integer $N$ is said to be superior $k$-highly composite  if there exists $\varepsilon>0$  such that 
\[\frac{d_k(n)}{n^{\varepsilon}}\leq \frac{d_k(N)}{N^{\varepsilon}}\hspace{1cm}\text{for all}\;n\in\mb{Z}^+.\]
\end{definition}
 
A superior $k$-highly composite number   is $k$-highly composite.  In the following, we study the superior $k$-highly composite numbers following the recent study of superior highly composite numbers given in \cite{Nicolas_22}.

For a fixed $\varepsilon>0$,  \eqref{20250731_1} implies that
\[\lim_{n\to\infty}\frac{d_k(n)}{n^{\varepsilon}}=0.\]
Hence, there exists a positive integer $N$ so that
\[\frac{d_k(n)}{n^{\varepsilon}}\leq \frac{d_k(N)}{N^{\varepsilon}}\hspace{1cm}\text{for all}\;n\in\mb{Z}^+.\] Such an integer $N$ is called a superior $k$-highly composite number   associated to  $\varepsilon$, and is denoted  as $N_{k,\varepsilon}$. In other words, for every $\varepsilon>0$, there exists a superior $k$-highly composite number   associated to  $\varepsilon$. To find $N_{k,\varepsilon}$ for a fixed $\varepsilon>0$, we recall the following lemma from \cite{Nicolas}.  

\begin{lemma}\label{20250717_1}
If $g(n)$ is a positive multiplicative   function such that  $\di\lim_{p^m\to\infty}g(p^m)=0$, then $g(n)$ is bounded and 
\[\max_{n\in \mathbb{Z}^+}g(n)=\prod_{p}\max_{m\geq 0} g(p^m).\]
\end{lemma}

The following theorem characterizes superior $k$-highly composite numbers, generalizing the $k=2$ case given in \cite{Nicolas}. The result has been stated in \cite{Duras} without proof, and the notation used in \cite{Duras} is also slightly different.
\begin{theorem}\label{20250717_2}
Given $\varepsilon>0$, let $N_{k,\varepsilon}$ be a superior $k$-highly composite number  associated to $\varepsilon$. Then
\[N_{k,\varepsilon}=\prod_{p\leq k^{1/\varepsilon}}p^{m(p,k,\varepsilon)},\]
where
\begin{align*}
m(p,k,\varepsilon)=\begin{cases}\di \left\lfloor \frac{k-1}{p^{\varepsilon}-1}\right\rfloor,\hspace{1cm} &\text{if} \;\;\di  \frac{k-1}{p^{\varepsilon}-1}\notin \mb{Z}^+,\\[2ex]\di \left\lfloor \frac{k-1}{p^{\varepsilon}-1}\right\rfloor\;\;\text{or}\;\;\left\lfloor \frac{k-1}{p^{\varepsilon}-1}\right\rfloor-1,\hspace{1cm} &\text{if} \;\; \di \frac{k-1}{p^{\varepsilon}-1}\in \mb{Z}^+\end{cases} 
\end{align*}
\end{theorem}
\begin{proof}
For $n\in\mb{Z}^+$, let
\[g(n)=\frac{d_k(n)}{n^{\varepsilon}}.\]
Then $g(n)$ is a positive multiplicative function, and by definition of superior $k$-highly composite numbers,
\[\max_{n\in\mb{Z}^+}g(n)=g(N_{k,\varepsilon}).\]
By Lemma \ref{20250717_1}, 
\[g(N_{k,\varepsilon})=\prod_{p}\max_{m\geq 0}g(p^{m}).\]
For a fixed prime $p$, when $m$ is a nonnegative integer,
\[g(p^m)=\frac{d_k(p^m)}{p^{m\varepsilon}}=\binom{m+k-1}{k-1}\frac{1} {p^{m\varepsilon}}.\]
Observe that  $g(p^0)=1$. If $p> k^{1/\varepsilon}$ and $m\geq 1$, then $p^{m\varepsilon}>k^m$. In this case, \eqref{dkpm} implies that
\[g(p^m)\leq \frac{k^m}{p^{m\varepsilon}}< 1.\] Hence, for $p> k^{1/\varepsilon}$,
\[\max_{m\geq 0}g(p^{m})=1.\]
Therefore,
\[N_{k,\varepsilon}=\prod_{p\leq k^{1/\varepsilon}}p^{m(p,k,\varepsilon)},\]where $m(p,k,\varepsilon)$ is a nonnegative integer such that
\[g(p^m)\leq g(p^{m(p,k,\varepsilon)})\hspace{1cm}\text{for all}\;m\geq 0.\]
 Now notice that $g(p^{m-1})\leq g(p^{m})$ if and only if
\[p^{\varepsilon}\leq \frac{m+k-1}{m},\]
if and only if
\[m\leq \frac{k-1}{p^{\varepsilon}-1}.\]
In fact, $g(p^{m-1})=g(p^m)$ if and only if
\[m=\frac{k-1}{p^{\varepsilon}-1}\in\mb{Z}^+.\]
The assertion of the theorem follows from this. \qedhere
\end{proof}

By Theorem \ref{20250717_2}, one can find a superior $k$-highly composite number $N_{k,\varepsilon}$  associated to any $\varepsilon>0$ by computing $m(p,k,\varepsilon)$ for all primes $p\leq k^{1/\varepsilon}$. For fixed $\varepsilon>0$, the largest superior $k$-highly composite number associated to  $\varepsilon$ is
\begin{equation}\label{20250731_2}\widetilde{N}_{k,\varepsilon}=\prod_{p\leq k^{1/\varepsilon}}p^{\widetilde{m}(p,k,\varepsilon)},\hspace{1cm}\text{where}\;\;\widetilde{m}(p,k,\varepsilon)=\left\lfloor\frac{k-1}{p^{\varepsilon}-1}\right\rfloor.\end{equation} 
 
For a fixed prime $p$, define the number  $\varepsilon[k,p;m]$,  $ m\in\mb{Z}^+$ by
\[\varepsilon[k,p;m]=\log_p \frac{m+k-1}{m}=\log_p\left(1+\frac{k-1}{m}\right).\]
Then  
\[\varepsilon[k,p;m+1]<\varepsilon[k,p;m]\hspace{1cm}\text{for all}\;m\in\mb{Z}^+,\] 
and if $\varepsilon[k,p;m+1]<\varepsilon\leq\varepsilon[k,p;m]$,
\[\left\lfloor\frac{k-1}{p^{\varepsilon}-1}\right\rfloor=m.\]
Let $\mathscr{A}_p(k)$ be the set 
\[\mathscr{A}_p(k)=\left\{\varepsilon[k,p;m]\, |\, m\in\mb{Z}^+ \right\}.\]
Then it is the set where the function 
\[\left\lfloor\frac{k-1}{p^{\varepsilon}-1}\right\rfloor,\]
as a function of $\varepsilon\in (0,\infty)$, has jumps.
Let
\[\mathscr{A}(k)=\bigcup_{p}\mathscr{A}_p(k).\]This is a countable subset of positive real numbers, the largest element being $\varepsilon[k,2;1]=\log_2k$. 
One can list down the elements of $\mathscr{A}(k)$ in decreasing order as $\varepsilon_1[k], \varepsilon_2[k], \ldots$.

If $\varepsilon\in (0,\infty)\setminus \mathscr{A}(k)$, Theorem \ref{20250717_2} says that there is a unique superior $k$-highly composite number   associated to $\varepsilon$ which is the number $\widetilde{N}_{k,\varepsilon}$ defined by \eqref{20250731_2}.
If  $\varepsilon>\varepsilon_1[k]$, then $\widetilde{N}_{k,\varepsilon}=1$.

If $\varepsilon\in\mathscr{A}(k)$, then there are more than one  superior $k$-highly composite numbers    associated to $\varepsilon$.

If  $p_1$ and $p_2$ are two distinct primes, the two sets  $\mathscr{A}_{p_1}(k)$ and $\mathscr{A}_{p_2}(k)$ are not necessarily disjoint.  For example, when $k=3$, 
\[\log_2\left(1+\frac{2}{2}\right)=\varepsilon[3,2;2]=\varepsilon[3,3;1]=\log_3\left(1+\frac{2}{1}\right)\] is in $\mathscr{A}_2(3)$ and $\mathscr{A}_{3}(3)$.
In fact, for any positive integer $r$, there exist $k\geq 2$ and $\varepsilon>0$ such that $\varepsilon\in\mathscr{A}_p(k)$ for at least $r$ prime numbers $p$. Take $r$ distinct prime numbers $p_1, \ldots, p_r$ and consider 
\[k=(p_1-1)\cdots (p_r-1)+1.\]
For $1\leq i\leq r$,
\[m_i=\frac{k-1}{p_i-1}\] is a positive integer and we find that
\[1+\frac{k-1}{m_i}=p_i.\]
Thus,
\[\varepsilon[k,p_1;m_1]=\cdots =\varepsilon[k,p_r;m_r]=1.\]

In \cite{Nicolas_22}, those $\varepsilon$ in the set $\mathscr{A}(k)$ which belongs to $\mathscr{A}_p(k)$ for more than one $p$ is said to be \emph{ extraordinary}. From our discussing above, there exist infinitely many $k$ which have extraordinary $\varepsilon$. 

In the following, we give generalizations of Lemma 3.6 and  Theorem 3.7 in \cite{Nicolas_22} to $k>2$.  
 
For a prime $p$, the largest element in $\mathscr{A}_p(k)$ is $\log_pk$. Hence, for any positive number $\varepsilon$, it can be in $\mathscr{A}_p(k)$ for finitely many primes $p$.
Now consider two consecutive numbers in the set $\mathscr{A}(k)$ given by $\varepsilon_{\alpha}$ and $\varepsilon_{\beta}$,  with $\varepsilon_{\alpha}<\varepsilon_{\beta}$.
Let $p_{\alpha_1}$, $p_{\alpha_2}$, $\ldots$, $p_{\alpha_r}$ be the list of all the primes $p$ so that $\varepsilon_{\alpha}\in \mathscr{A}_{p}(k)$, and let 
$p_{\beta_1}$, $p_{\beta_2}$, $\ldots$, $p_{\beta_s}$ be the list of all the primes $p$ so that $\varepsilon_{\beta}\in \mathscr{A}_{p}(k)$. Then there are positive integers $m_{\alpha_1}$, $m_{\alpha_2}$, $\ldots$, $m_{\alpha_r}$, $m_{\beta_1}$, $m_{\beta_2}$, $\ldots$, $m_{\beta_s}$ such that
\[\varepsilon_{\alpha}=\varepsilon[k,p_{\alpha_1};m_{\alpha_1}]=\cdots=\varepsilon[k,p_{\alpha_r};m_{\alpha_r}],\]
\[\varepsilon_{\beta}=\varepsilon[k,p_{\beta_1};m_{\beta_1}]=\cdots=\varepsilon[k,p_{\beta_s};m_{\beta_s}].\]
Theorem \ref{20250717_2} says that there are $2^r$ superior $k$-highly composite numbers   associated to $\varepsilon_{\alpha}$, the largest one being $\widetilde{N}_{k,\varepsilon_{\alpha}}$ defined by \eqref{20250731_2}. Each of the $N_{k,\varepsilon_{\alpha}}$  can be written as
\[\widetilde{N}_{k,\varepsilon_{\alpha}}\prod_{i=1}^r p_{\alpha_i}^{-a_i},\]where
$(a_1, a_2, \ldots, a_r)$ is an $r$-tuple with $a_i=0$ or $1$. 
Similarly, all the superior $k$-highly composite numbers   associated to $\varepsilon_{\beta}$ are 
\[\widetilde{N}_{k,\varepsilon_{\beta}}\prod_{j=1}^s p_{\beta_j}^{-b_j},\]where
$(b_1, b_2, \ldots, b_s)$ is a $s$-tuple with $b_j=0$ or $1$. For $\varepsilon\in (\varepsilon_{\alpha},\varepsilon_{\beta})$,
$\widetilde{m}(p,k,\varepsilon)$ is a constant for all primes $p$. In fact, if $p\neq p_{\alpha_i}$ for all $1\leq i\leq r$, then
\[\widetilde{m}(p,k,\varepsilon_{\beta})=\widetilde{m}(p,k,\varepsilon)=\widetilde{m}(p,k,\varepsilon_{\alpha}).\]
If $p=p_{\alpha_i}$ for some $1\leq i\leq r$, then
\[\widetilde{m}(p_{\alpha_i},k,\varepsilon_{\beta})=\widetilde{m}(p_{\alpha_i},k,\varepsilon)=\widetilde{m}(p_{\alpha_i},k,\varepsilon_{\alpha})-1.\]
Therefore, for all $\varepsilon\in (\varepsilon_{\alpha},\varepsilon_{\beta})$, the unique superior $k$-highly composite number $\widetilde{N}_{k,\varepsilon}$   associated to $\varepsilon$ is
\[\widetilde{N}_{k,\varepsilon}=\widetilde{N}_{k,\varepsilon_{\beta}}=\widetilde{N}_{k,\varepsilon_{\alpha}}\prod_{i=1}^r p_{\alpha_i}^{-1}.\] Note that  $\widetilde{N}_{k,\varepsilon_{\beta}}$ is one of  the superior $k$-highly composite numbers   associated to $\varepsilon_{\alpha}$. If $\varepsilon_{\alpha}$ is not extraordinary, then $\widetilde{N}_{k,\varepsilon_{\alpha}}$ and $\widetilde{N}_{k,\varepsilon_{\beta}}$ are the only two superior $k$-highly composite numbers associated to $\varepsilon_{\alpha}$.

 This gives a way to produce all the superior $k$-highly composite numbers. For  $\varepsilon_1[k]=\log_2k$, it is obvious that it is only contained in $\mathscr{A}_2(k)$. Hence, the superior $k$-highly composite numbers  associated to $\varepsilon_1[k]=\log_2k$ are 1 and 2. Using the discussion above, one can write a computer algorithm to generate all the superior $k$-highly composite numbers  associated to $\varepsilon$ with $\varepsilon$ larger than or equal to a certain $\varepsilon_0$. 
 
 \vfill\pagebreak
 In Table \ref{tab5}, we list down the first 20 $\varepsilon$'s in the set $\mathscr{A}(2)$ and the corresponding superior $2$-highly composite numbers $\widetilde{N}_{2,\varepsilon}$. In Table \ref{tab6}, we list down the range of  $\varepsilon$ and the corresponding superior $2$-highly composite numbers $N_{2,\varepsilon}$. 
 
 \begin{table}[h]
 \caption{\label{tab5}The first 20 $\varepsilon$'s in $\mathscr{A}(2)$ and the corresponding $\widetilde{N}_{2,\varepsilon}$.}
 \begin{tabular}{|c|c|c|l|}

\hline
 $\varepsilon[2,p;m]$	&	$p$	& $m$		 & $\widetilde{N}_{2,\varepsilon}$ \\[1ex]\hline
1	&	2	&	1	&$	2=2	$\\\hline
0.6309	&	3	&	1	&$	6=2\times 3	$\\\hline
0.5850	&	2	&	2	&$	12=2^2\times 3	$\\\hline
0.4307	&	5	&	1	&$	60=2^2\times 3\times 5	$\\\hline
0.4150	&	2	&	3	&$	120=2^3\times 3\times 5	$\\\hline
0.3691	&	3	&	2	&$	360=2^3\times 3^2\times 5	$\\\hline
0.3562	&	7	&	1	&$	2520=2^3\times 3^2\times 5\times 7	$\\\hline
0.3219	&	2	&	4	&$	5040=2^4\times 3^2\times 5\times 7	$\\\hline
0.2891	&	11	&	1	&$	55440=2^4\times 3^2\times 5\times 7\times 11	$\\\hline
0.2702	&	13	&	1	&$	720720=2^4\times 3^2\times 5\times 7\times 11\times 13	$\\\hline
0.2630	&	2	&	5	&$	1441440=2^5\times 3^2\times 5\times 7\times 11\times 13	$\\\hline
0.2619	&	3	&	3	&$	4324320=2^5\times 3^3\times 5\times 7\times 11\times 13	$\\\hline
0.2519	&	5	&	2	&$	21621600=2^5\times 3^3\times 5^2\times 7\times 11\times 13	$\\\hline
0.2447	&	17	&	1	&$	367567200=2^5\times 3^3\times 5^2\times 7\times 11\times 13\times 17	$\\\hline
0.2354	&	19	&	1	&$	6983776800=2^5\times 3^3\times 5^2\times 7\times 11\times 13\times 17\times 19	$\\\hline
0.2224	&	2	&	6	&$	13967553600=2^6\times 3^3\times 5^2\times 7\times 11\times 13\times 17\times 19	$\\\hline
0.2211	&	23	&	1	&$	321253732800=2^6\times 3^3\times 5^2\times 7\times 11\times 13\times 17\times 19\times 23	$\\\hline
0.2084	&	7	&	2	&$	2248776129600=2^6\times 3^3\times 5^2\times 7^2\times 11\times 13\times 17\times 19\times 23	$\\\hline
0.2058	&	29	&	1	&$	65214507758400=2^6\times 3^3\times 5^2\times 7^2\times 11\times 13\times 17\times 19\times 23\times 29	$\\\hline
0.2031	&	3	&	4	&$	195643523275200=2^6\times 3^4\times 5^2\times 7^2\times 11\times 13\times 17\times 19\times 23\times 29	$\\\hline

 \end{tabular}
 \end{table}
 
 \vfill\pagebreak
 ~
 \begin{table}
 \caption{\label{tab6}The range of $\varepsilon$ and the corresponding $N_{2,\varepsilon}$.}
 \begin{tabular}{|c|l|}
 \hline
 Range of $\varepsilon$	& $N_{k,\varepsilon}$	\\\hline
$	\varepsilon>1	$&$	1	$\\\hline
$	\varepsilon=1	$&$	1, 2	$\\\hline
$	0.6309<\varepsilon<1	$&$	2	$\\\hline
$	\varepsilon=0.6309	$&$	2, 6	$\\\hline
$	0.5850<\varepsilon<0.6309	$&$	6	$\\\hline
$	\varepsilon=0.5850	$&$	6, 12	$\\\hline
$	0.4307<\varepsilon<0.5850	$&$	12	$\\\hline
$	\varepsilon=0.4307	$&$	12, 60	$\\\hline
$	0.4150<\varepsilon<0.4307	$&$	60	$\\\hline
$	\varepsilon=0.4150	$&$	60, 120	$\\\hline
$	\hsp 0.3691<\varepsilon<\hsp 0.4150	$&$	120	$\\\hline
$	\varepsilon=0.3691	$&$	120, 360	$\\\hline
$	0.3562<\varepsilon<0.3691	$&$	360	$\\\hline
$	\varepsilon=0.3562	$&$	360, 2520	$\\\hline
$	0.3219<\varepsilon<0.3562	$&$	2520	$\\\hline
$	\varepsilon=0.3219	$&$	2520, 5040	$\\\hline
$	0.2891<\varepsilon<0.3219	$&$	5040	$\\\hline
$	\varepsilon=0.2891	$&$	5040, 55440	$\\\hline
$	0.2702<\varepsilon<0.2891	$&$	55440	$\\\hline
$	\varepsilon=0.2702	$&$	55440, 720720	$\\\hline
$	0.2630<\varepsilon<0.2702	$&$	720720	$\\\hline
$	\varepsilon=0.2630	$&$	720720, 1441440	$\\\hline
$	0.2619<\varepsilon<0.2630	$&$	1441440	$\\\hline
$	\varepsilon=0.2619	$&$	1441440, 4324320	$\\\hline
$	0.2519<\varepsilon<0.2619	$&$	4324320	$\\\hline
$	\varepsilon=0.2519	$&$	4324320, 21621600	$\\\hline
$	0.2447<\varepsilon<0.2519	$&$	21621600	$\\\hline
$	\varepsilon=0.2447	$&$	21621600, 367567200	$\\\hline
$	0.2354<\varepsilon<0.2447	$&$	367567200	$\\\hline
$	\varepsilon=0.2354	$&$	367567200, 6983776800	$\\\hline
$	0.2224<\varepsilon<0.2354	$&$	6983776800	$\\\hline
$	\varepsilon=0.2224	$&$	6983776800, 13967553600	$\\\hline
$	0.2211<\varepsilon<0.2224	$&$	13967553600	$\\\hline
$	\varepsilon=0.2211	$&$	13967553600, 321253732800	$\\\hline
$	0.2084<\varepsilon<0.2211	$&$	321253732800	$\\\hline
$	\varepsilon=0.2084	$&$	321253732800, 2248776129600	$\\\hline

\end{tabular}
\end{table}

\vfill\pagebreak
 In Table \ref{tab7}, we list down the first 20 $\varepsilon$'s in the set $\mathscr{A}(3)$ and the corresponding superior $3$-highly composite numbers $\widetilde{N}_{3,\varepsilon}$. In Table \ref{tab8}, we list down the range of  $\varepsilon$ and the corresponding superior $3$-highly composite numbers $N_{3,\varepsilon}$. 
 
 \begin{table}[h]
 \caption{\label{tab7}The first 20 $\varepsilon$'s in $\mathscr{A}(3)$ and the corresponding $\widetilde{N}_{3,\varepsilon}$.}
 \begin{tabular}{|c|c|c|l|}
 \hline
 $\varepsilon[3,p;m]$	&	$p$	& $m$		 & $\widetilde{N}_{3,\varepsilon}$ \\[1ex]\hline
1.5850	&	2	&	1	&$	2=2	$\\\hline
\multirow{2}{2.8em}{1}	&	2	&	2	& 	\multirow{2}{10em}{$12=2^2\times 3	$}\\
 	&	3	&	1	&	 \\\hline
0.7370	&	2	&	3	&$	24=2^3\times 3	$\\\hline
0.6826	&	5	&	1	&$	120=2^3\times 3\times 5	$\\\hline
0.6309	&	3	&	2	&$	360=2^3\times 3^2\times 5	$\\\hline
0.5850	&	2	&	4	&$	720=2^4\times 3^2\times 5	$\\\hline
0.5646	&	7	&	1	&$	5040=2^4\times 3^2\times 5\times 7	$\\\hline
0.4854	&	2	&	5	&$	10080=2^5\times 3^2\times 5\times 7	$\\\hline
0.4650	&	3	&	3	&$	30240=2^5\times 3^3\times 5\times 7	$\\\hline
0.4582	&	11	&	1	&$	332640=2^5\times 3^3\times 5\times 7\times 11	$\\\hline
0.4307	&	5	&	2	&$	1663200=2^5\times 3^3\times 5^2\times 7\times 11	$\\\hline
0.4283	&	13	&	1	&$	21621600=2^5\times 3^3\times 5^2\times 7\times 11\times 13	$\\\hline
0.4150	&	2	&	6	&$	43243200=2^6\times 3^3\times 5^2\times 7\times 11\times 13	$\\\hline
0.3878	&	17	&	1	&$	735134400=2^6\times 3^3\times 5^2\times 7\times 11\times 13\times 17	$\\\hline
0.3731	&	19	&	1	&$	13967553600=2^6\times 3^3\times 5^2\times 7\times 11\times 13\times 17\times 19	$\\\hline
0.3691	&	3	&	4	&$	41902660800=2^6\times 3^4\times 5^2\times 7\times 11\times 13\times 17\times 19	$\\\hline
0.3626	&	2	&	7	&$	83805321600=2^7\times 3^4\times 5^2\times 7\times 11\times 13\times 17\times 19	$\\\hline
0.3562	&	7	&	2	&$	586637251200=2^7\times 3^4\times 5^2\times 7^2\times 11\times 13\times 17\times 19	$\\\hline
0.3504	&	23	&	1	&$	13492656777600=2^7\times 3^4\times 5^2\times 7^2\times 11\times 13\times 17\times 19\times 23	$\\\hline
0.3263	&	29	&	1	&$	391287046550400=2^7\times 3^4\times 5^2\times 7^2\times 11\times 13\times 17\times 19\times 23\times 29	$\\\hline

 \end{tabular}
 \end{table}
 
 \vfill\pagebreak
 ~
 \begin{table}
 \caption{\label{tab8}The range of $\varepsilon$ and the corresponding $N_{3,\varepsilon}$.}
 \begin{tabular}{|c|l|}
 \hline
 Range of $\varepsilon$	& $N_{3,\varepsilon}$	\\\hline
$	\varepsilon>1.5850	$&$	1	$\\\hline
$	\varepsilon=1.5850	$&$	1, 2	$\\\hline
$	1<\varepsilon<1.5850	$&$	2	$\\\hline
$	\varepsilon=1	$&$	2, 4, 6, 12	$\\\hline
$	0.7370<\varepsilon<1	$&$	12	$\\\hline
$	\varepsilon=0.7370	$&$	12, 24	$\\\hline
$	\hsp 0.6826<\varepsilon<0.7370	\hsp $&$	24	$\\\hline
$	\varepsilon=0.6826	$&$	24, 120	$\\\hline
$	0.6309<\varepsilon<0.6826	$&$	120	$\\\hline
$	\varepsilon=0.6309	$&$	120, 360	$\\\hline
$	0.5850<\varepsilon<0.6309	$&$	360	$\\\hline
$	\varepsilon=0.5850	$&$	360, 720	$\\\hline
$	0.5646<\varepsilon<0.5850	$&$	720	$\\\hline
$	\varepsilon=0.5646	$&$	720, 5040	$\\\hline
$	0.4854<\varepsilon<0.5646	$&$	5040	$\\\hline
$	\varepsilon=0.4854	$&$	5040, 10080	$\\\hline
$	0.4650<\varepsilon<0.4854	$&$	10080	$\\\hline
$	\varepsilon=0.4650	$&$	10080, 30240	$\\\hline
$	0.4582<\varepsilon<0.4650	$&$	30240	$\\\hline
$	\varepsilon=0.4582	$&$	30240, 332640	$\\\hline
$	0.4307<\varepsilon<0.4582	$&$	332640	$\\\hline
$	\varepsilon=0.4307	$&$	332640, 1663200	$\\\hline
$	0.4283<\varepsilon<0.4307	$&$	1663200	$\\\hline
$	\varepsilon=0.4283	$&$	1663200, 21621600	$\\\hline
$	0.4150<\varepsilon<0.4283	$&$	21621600	$\\\hline
$	\varepsilon=0.4150	$&$	21621600, 43243200	$\\\hline
$	0.3878<\varepsilon<0.4150	$&$	43243200	$\\\hline
$	\varepsilon=0.3878	$&$	43243200, 735134400	$\\\hline
$	0.3731<\varepsilon<0.3878	$&$	735134400	$\\\hline
$	\varepsilon=0.3731	$&$	735134400, 13967553600	$\\\hline
$	0.3691<\varepsilon<0.3731	$&$	13967553600	$\\\hline
$	\varepsilon=0.3691	$&$	13967553600, 41902660800	$\\\hline
$	0.3626<\varepsilon<0.3691	$&$	41902660800	$\\\hline
$	\varepsilon=0.3626	$&$	41902660800, 83805321600	$\\\hline
$	0.3562<\varepsilon<0.3626	$&$	83805321600	$\\\hline
$	\varepsilon=0.3562	$&$	83805321600, 586637251200	$\\\hline

\end{tabular}
\end{table}

 \vfill\pagebreak
 In Table \ref{tab9}, we list down the first 20 $\varepsilon$'s in the set $\mathscr{A}(4)$ and the corresponding superior $4$-highly composite numbers $\widetilde{N}_{4,\varepsilon}$. In Table \ref{tab10}, we list down the range of  $\varepsilon$ and the corresponding superior $4$-highly composite numbers $N_{4,\varepsilon}$. 
 
 \begin{table}[h]
 \caption{\label{tab9}The first 20 $\varepsilon$'s in $\mathscr{A}(4)$ and the corresponding $\widetilde{N}_{4,\varepsilon}$.}
 \begin{tabular}{|c|c|c|l|}
 \hline
 $\varepsilon[4,p;m]$	&	$p$	& $m$		 & $\widetilde{N}_{4,\varepsilon}$ \\[1ex]\hline
2	&	2	&	1	&$	2=2	$\\\hline
1.3219	&	2	&	2	&$	4=2^2	$\\\hline
1.2619	&	3	&	1	&$	12=2^2\times 3	$\\\hline
1	&	2	&	3	&$	24=2^3\times 3	$\\\hline
0.8614	&	5	&	1	&$	120=2^3\times 3\times 5	$\\\hline
0.8340	&	3	&	2	&$	360=2^3\times 3^2\times 5	$\\\hline
0.8074	&	2	&	4	&$	720=2^4\times 3^2\times 5	$\\\hline
0.7124	&	7	&	1	&$	5040=2^4\times 3^2\times 5\times 7	$\\\hline
0.6781	&	2	&	5	&$	10080=2^5\times 3^2\times 5\times 7	$\\\hline
0.6309	&	3	&	3	&$	30240=2^5\times 3^3\times 5\times 7	$\\\hline
0.5850	&	2	&	6	&$	60480=2^6\times 3^3\times 5\times 7	$\\\hline
0.5781	&	11	&	1	&$	665280=2^6\times 3^3\times 5\times 7\times 11	$\\\hline
0.5693	&	5	&	2	&$	3326400=2^6\times 3^3\times 5^2\times 7\times 11	$\\\hline
0.5405	&	13	&	1	&$	43243200=2^6\times 3^3\times 5^2\times 7\times 11\times 13	$\\\hline
0.5146	&	2	&	7	&$	86486400=2^7\times 3^3\times 5^2\times 7\times 11\times 13	$\\\hline
0.5094	&	3	&	4	&$	259459200=2^7\times 3^4\times 5^2\times 7\times 11\times 13	$\\\hline
0.4893	&	17	&	1	&$	4410806400=2^7\times 3^4\times 5^2\times 7\times 11\times 13\times 17	$\\\hline
0.4709	&	7	&	2	&$	30875644800=2^7\times 3^4\times 5^2\times 7^2\times 11\times 13\times 17	$\\\hline
0.4708	&	19	&	1	&$	586637251200=2^7\times 3^4\times 5^2\times 7^2\times 11\times 13\times 17\times 19	$\\\hline
0.4594	&	2	&	8	&$	1173274502400=2^8\times 3^4\times 5^2\times 7^2\times 11\times 13\times 17\times 19	\hsp$\\\hline

 \end{tabular}
 \end{table}
 
 \vfill\pagebreak
 ~
 \begin{table}
 \caption{\label{tab10}The range of $\varepsilon$ and the corresponding $N_{4,\varepsilon}$.}
 \begin{tabular}{|c|l|}
 \hline
 Range of $\varepsilon$	& $N_{4,\varepsilon}$	\\\hline
$	 \varepsilon>2	$&$	1	$\\\hline
$	\varepsilon=2	$&$	1, 2	$\\\hline
$	1.3219<\varepsilon<2	$&$	2	$\\\hline
$	\varepsilon=1.3219	$&$	2, 4	$\\\hline
$	1.2619<\varepsilon<1.3219	$&$	4	$\\\hline
$	\varepsilon=1.2619	$&$	4, 12	$\\\hline
$	1<\varepsilon<1.2619	$&$	12	$\\\hline
$	\varepsilon=1	$&$	12, 24	$\\\hline
$	0.8614<\varepsilon<1	$&$	24	$\\\hline
$	\varepsilon=0.8614	$&$	24, 120	$\\\hline
$	0.8340<\varepsilon<0.8614	$&$	120	$\\\hline
$	\varepsilon=0.8340	$&$	120, 360	$\\\hline
$	0.8074<\varepsilon<0.8340	$&$	360	$\\\hline
$	\varepsilon=0.8074	$&$	360, 720	$\\\hline
$	0.7124<\varepsilon<0.8074	$&$	720	$\\\hline
$	\varepsilon=0.7124	$&$	720, 5040	$\\\hline
$	0.6781<\varepsilon<0.7124	$&$	5040	$\\\hline
$	\varepsilon=0.6781	$&$	5040, 10080	$\\\hline
$	0.6309<\varepsilon<0.6781	$&$	10080	$\\\hline
$	\varepsilon=0.6309	$&$	10080, 30240	$\\\hline
$	0.5850<\varepsilon<0.6309	$&$	30240	$\\\hline
$	\varepsilon=0.5850	$&$	30240, 60480	$\\\hline
$	0.5781<\varepsilon<0.5850	$&$	60480	$\\\hline
$	\varepsilon=0.5781	$&$	60480, 665280	$\\\hline
$	0.5693<\varepsilon<0.5781	$&$	665280	$\\\hline
$	\varepsilon=0.5693	$&$	665280, 3326400	$\\\hline
$	0.5405<\varepsilon<0.5693	$&$	3326400	$\\\hline
$	\varepsilon=0.5405	$&$	3326400, 43243200	$\\\hline
$	0.5146<\varepsilon<0.5405	$&$	43243200	$\\\hline
$	\varepsilon=0.5146	$&$	43243200, 86486400	$\\\hline
$	0.5094<\varepsilon<0.5146	$&$	86486400	$\\\hline
$	\varepsilon=0.5094	$&$	86486400, 259459200	$\\\hline
$\hsp	0.4893<\varepsilon<0.5094\hsp	$&$	259459200	$\\\hline
$	\varepsilon=0.4893	$&$	259459200, 4410806400	$\\\hline
$	0.4709<\varepsilon<0.4893	$&$	4410806400	$\\\hline
$	\varepsilon=0.4709	$&$	4410806400, 30875644800	$\\\hline

\end{tabular}
\end{table}

\vfill\pagebreak

 In Table \ref{tab11}, we list down the first 20 $\varepsilon$'s in the set $\mathscr{A}(5)$ and the corresponding superior $5$-highly composite numbers $\widetilde{N}_{5,\varepsilon}$. In Table \ref{tab12}, we list down the range of  $\varepsilon$ and the corresponding superior $5$-highly composite numbers $N_{5,\varepsilon}$. 
 
 \begin{table}[h]
 \caption{\label{tab11}The first 20 $\varepsilon$'s in $\mathscr{A}(5)$ and the corresponding $\widetilde{N}_{5,\varepsilon}$.}
 \begin{tabular}{|c|c|c|l|}
 \hline
 $\varepsilon[5,p;m]$	&	$p$	& $m$		 & $\widetilde{N}_{5,\varepsilon}$ \\[1ex]\hline
2.3219	&	2	&	1	&$	2=2	$\\\hline
1.5850	&	2	&	2	&$	4=2^2	$\\\hline
1.4650	&	3	&	1	&$	12=2^2\times 3	$\\\hline
1.2224	&	2	&	3	&$	24=2^3\times 3	$\\\hline
\multirow{3}{1em}{1}	&	2	&	4	&\multirow{3}{10em}{$	720=2^4\times 3^2\times 5	$}\\ 
 &	3	&	2	& \\ 
 	&	5	&	1	& \\\hline
0.8480	&	2	&	5	&$	1440=2^5\times 3^2\times 5	$\\\hline
0.8271	&	7	&	1	&$	10080=2^5\times 3^2\times 5\times 7	$\\\hline
0.7712	&	3	&	3	&$	30240=2^5\times 3^3\times 5\times 7	$\\\hline
0.7370	&	2	&	6	&$	60480=2^6\times 3^3\times 5\times 7	$\\\hline
0.6826	&	5	&	2	&$	302400=2^6\times 3^3\times 5^2\times 7	$\\\hline
0.6712	&	11	&	1	&$	3326400=2^6\times 3^3\times 5^2\times 7\times 11	$\\\hline
0.6521	&	2	&	7	&$	6652800=2^7\times 3^3\times 5^2\times 7\times 11	$\\\hline
0.6309	&	3	&	4	&$	19958400=2^7\times 3^4\times 5^2\times 7\times 11	$\\\hline
0.6275	&	13	&	1	&$	259459200=2^7\times 3^4\times 5^2\times 7\times 11\times 13	$\\\hline
0.5850	&	2	&	8	&$	518918400=2^8\times 3^4\times 5^2\times 7\times 11\times 13	$\\\hline
0.5681	&	17	&	1	&$	8821612800=2^8\times 3^4\times 5^2\times 7\times 11\times 13\times 17	$\\\hline
0.5646	&	7	&	2	&$	61751289600=2^8\times 3^4\times 5^2\times 7^2\times 11\times 13\times 17	$\\\hline
0.5466	&	19	&	1	&$	1173274502400=2^8\times 3^4\times 5^2\times 7^2\times 11\times 13\times 17\times 19	$\\\hline
0.5350 	&	3	&	5	&$	3519823507200=2^8\times 3^5\times 5^2\times 7^2\times 11\times 13\times 17\times 19	$\\\hline
0.5305 	&	2	&	9	&$	7039647014400=2^9\times 3^5\times 5^2\times 7^2\times 11\times 13\times 17\times 19\hsp	$\\\hline

 \end{tabular}
 \end{table}
 
 \vfill\pagebreak
 
 \begin{table}
 \caption{\label{tab12}The range of $\varepsilon$ and the corresponding $N_{5,\varepsilon}$.}
 \begin{tabular}{|c|l|}
 \hline
 Range of $\varepsilon$	& $N_{5,\varepsilon}$	\\\hline
$	\varepsilon>2.3219	$&$	1	$\\\hline
$	\varepsilon=2.3219	$&$	1, 2	$\\\hline
$	1.5850<\varepsilon<2.3219	$&$	2	$\\\hline
$	\varepsilon=1.5850	$&$	2, 4	$\\\hline
$	1.4650<\varepsilon<1.5850	$&$	4	$\\\hline
$	\varepsilon=1.4650	$&$	4, 12	$\\\hline
$	1.2224<\varepsilon<1.4650	$&$	12	$\\\hline
$	\varepsilon=1.2224	$&$	12, 24	$\\\hline
$	1<\varepsilon<1.2224	$&$	24	$\\\hline
$	\varepsilon=1	$&$	24, 48, 72, 120, 144, 240, 360, 720	$\\\hline
$	0.8480<\varepsilon<1	$&$	720	$\\\hline
$	\varepsilon=0.8480	$&$	720, 1440	$\\\hline
$	0.8271<\varepsilon<0.8480	$&$	1440	$\\\hline
$	\varepsilon=0.8271	$&$	1440, 10080	$\\\hline
$	0.7712<\varepsilon<0.8271	$&$	10080	$\\\hline
$	\varepsilon=0.7712	$&$	10080, 30240	$\\\hline
$	0.7370<\varepsilon<0.7712	$&$	30240	$\\\hline
$	\varepsilon=0.7370	$&$	30240, 60480	$\\\hline
$	0.6826<\varepsilon<0.7370	$&$	60480	$\\\hline
$	\varepsilon=0.6826	$&$	60480, 302400	$\\\hline
$	0.6712<\varepsilon<0.6826	$&$	302400	$\\\hline
$	\varepsilon=0.6712	$&$	302400, 3326400	$\\\hline
$	0.6521<\varepsilon<0.6712	$&$	3326400	$\\\hline
$	\varepsilon=0.6521	$&$	3326400, 6652800	$\\\hline
$\hsp	0.6309<\varepsilon<0.6521\hsp	$&$	6652800	$\\\hline
$	\varepsilon=0.6309	$&$	6652800, 19958400	$\\\hline
$	0.6275<\varepsilon<0.6309	$&$	19958400	$\\\hline
$	\varepsilon=0.6275	$&$	19958400, 259459200	$\\\hline
$	0.5850<\varepsilon<0.6275	$&$	259459200	$\\\hline
$	\varepsilon=0.5850	$&$	259459200, 518918400	$\\\hline
$	0.5681<\varepsilon<0.5850	$&$	518918400	$\\\hline
$	\varepsilon=0.5681	$&$	518918400, 8821612800	$\\\hline
$	0.5646<\varepsilon<0.5681	$&$	8821612800	$\\\hline
$	\varepsilon=0.5646	$&$	8821612800, 61751289600	$\\\hline
$	0.5466<\varepsilon<0.5646	$&$	61751289600	$\\\hline
$	\varepsilon=0.5466	$&$	61751289600, 1173274502400	$\\\hline

\end{tabular}
\end{table}

 The following theorem gives an expression of $\widetilde{N}_{k,\varepsilon}$   in terms of a well known arithmetic function. 
 \begin{theorem}\label{20250805_3}
 Let $\varepsilon$ be a positive number, and let $\widetilde{N}_{k,\varepsilon}$ be the largest superior $k$-highly composite number  associated to $\varepsilon$ defined by \eqref{20250731_2}. For $m\in\mathbb{Z}^+$, let
 \begin{equation}\label{20250805_6}x_{m}=\left(1+\frac{k-1}{m}\right)^{1/\varepsilon}.\end{equation}
 Then
 \begin{equation}\label{20250805_2}\log \widetilde{N}_{k,\varepsilon}=\sum_{m=1}^{\infty}\vartheta\left(x_m\right),\end{equation}
 where $\vartheta(x)$ is the Chebyshev's theta function.
 \end{theorem}
 \begin{proof}
 Notice that since $x_m$ decreases to $0$ as $m\to\infty$, and $\vartheta(x)=0$ for $x<2$, the right hand side of \eqref{20250805_2} is a finite sum. 
 By definition,
 \begin{align*}
 \log\widetilde{N}_{k,\varepsilon}=\sum_{p\leq k^{1/\varepsilon}}\widetilde{m}(p,k,\varepsilon)\log p,
 \end{align*}where
 \[\widetilde{m}(p,k,\varepsilon)=\left\lfloor\frac{k-1}{p^{\varepsilon}-1}\right\rfloor.\]
 Notice that for $m\in\mb{Z}^+$, $\widetilde{m}(p,k,\varepsilon)=m$ if and only if
 \[x_{m+1}=\left(1+\frac{k-1}{m+1}\right)^{1/\varepsilon}<p\leq \left(1+\frac{k-1}{m}\right)^{1/\varepsilon}=x_m.\]
 Therefore,
 \begin{align*}\log \widetilde{N}_{k,\varepsilon}&=\sum_{m=1}^{\infty}\sum_{x_{m+1}<p\leq x_m}m\log p\\
 &= \sum_{m=1}^{\infty}m\Bigl(\vartheta\left(x_m\right)-\vartheta\left(x_{m+1}\right)\Bigr)\\
 &= \sum_{m=1}^{\infty}m \,\vartheta\left(x_m\right)- \sum_{m=1}^{\infty}(m-1) \;\vartheta\left(x_m\right) \\
 &= \sum_{m=1}^{\infty} \vartheta\left(x_m\right).\hfill\qedhere\end{align*}

 \end{proof}
 
  The following theorem gives a lower bound of the superior $k$-highly composite number $\widetilde{N}_{k,\varepsilon}$  associated to $\varepsilon$. 
 \begin{theorem}\label{20250803_1}
 Let $\varepsilon$ be a positive number, and let $\widetilde{N}_{k,\varepsilon}$ be the superior $k$-highly composite number  associated to $\varepsilon$ defined by \eqref{20250731_2}. Then
 \[\log \widetilde{N}_{k,\varepsilon}\geq \psi(k^{1/\varepsilon}),\]
 where $\psi(x)$ is the Chebyshev's psi function.
 \end{theorem}
 \begin{proof}
 By Theorem \ref{20250805_3},
 \begin{align*}\log  \widetilde{N}_{k,\varepsilon} = \sum_{m=1}^{\infty} \vartheta\left(x_m\right).\end{align*}
 Now for any positive integer $m$ and any positive number $a$, $(1+a)^m\geq 1+am$. Hence,
 \[ \left(1+\frac{k-1}{m}\right)^{m}\geq 1+k-1=k=x_1^{\varepsilon}.\]
 Therefore, 
 \[x_m=\left(1+\frac{k-1}{m}\right)^{1/\varepsilon}\geq x_1^{1/m}.\]
 It follows that 
 \[  \vartheta\left(x_m\right)\geq \vartheta(x_1^{1/m}).\]
 Thus,
 \[\log \widetilde{N}_{k,\varepsilon}\geq \sum_{m=1}^{\infty} \vartheta(x_1^{1/m})=\psi(x_1)=\psi\left(k^{1/\varepsilon}\right). \qedhere\] 
 \end{proof}
For $k\geq 2$, define
\begin{equation}\label{20250720_6}\varepsilon_0(k)=\frac{\di\log\frac{k+1}{2}}{\log 8}.\end{equation}
The following theorem  gives an upper bound to a superior $k$-highly composite number.

\begin{theorem}\label{20250718_1}
  For $\varepsilon>0$, let $N_{k,\varepsilon}$ be a superior $k$-highly composite number  associated with $\varepsilon$. Then 
\[\log N_{k,\varepsilon}\leq c_0(k)k^{1/\varepsilon},\]
where
\begin{equation}\label{20250720_7_2}c_0(k)=\frac{k-1}{   e\di\log\frac{2k}{k+1}}   +\omega_2,\hspace{1cm}\text{with}\;\; \omega_2=1.00001.\end{equation}
 If 
$\varepsilon\leq \varepsilon_0(k)$, 
then we have a better result
\[\log N_{k,\varepsilon}\leq c_1(k)k^{1/\varepsilon},\]
where
\begin{equation}\label{20250720_7}c_1(k)=\frac{k-1}{ 2 e\di\log\frac{2k}{k+1}}   +\omega_2\geq \omega_2>1.\end{equation}
\end{theorem}
\begin{proof} Let 
\[ x_1=k^{1/\varepsilon},\hspace{1cm} x_2=\left(\frac{k+1}{2}\right)^{1/\varepsilon}.\]
Then $1\leq x_2\leq x_1$.
From the proof of Theorem \ref{20250805_3}, we have
\begin{align*}
\log N_{k,\varepsilon}\leq \log\widetilde{N}_{k,\varepsilon} =\sum_{p\leq x_1} \widetilde{m}(p,k,\varepsilon)\log p= \sum_{p\leq x_2}  \widetilde{m}(p,k,\varepsilon)\log p+\sum_{ x_2<p\leq x_1}\log p.
\end{align*}Now by mean value theorem,
\[\widetilde{m}(p,k,\varepsilon)\leq\frac{k-1}{p^{\varepsilon}-1}\leq \frac{k-1}{\varepsilon\log p}.\]
Therefore,
\begin{align*}
\log N_{k,\varepsilon}&\leq  \frac{k-1}{\varepsilon} \sum_{p\leq x_2} 1+\sum_{ x_2<p\leq x_1}\log p\leq \frac{k-1}{\varepsilon} \pi\left(x_2\right)+\vartheta\left(x_1\right).
\end{align*}
We have the trivial bound $\pi(x_2)\leq x_2$. On the other hand,
  Lemma \ref{20250723_3} gives
\[\vartheta\left(x_1\right)\leq \omega_2 x_1, \hspace{1cm}\text{where}\;\;\omega_2=1.00001.\]
Hence,
\[\log N_{k,\varepsilon}\leq k^{\frac{1}{\varepsilon}}\left[\frac{k-1}{ \varepsilon}  \left(\frac{k+1}{2k}\right)^{\frac{1}{\varepsilon}} +\omega_2\right].\]
Let \[t=\frac{1}{\varepsilon}\hspace{1cm}\text{and}\hspace{1cm}a=\log\frac{2k}{k+1}.\] The function $h(t)=te^{-at}$ has a maximum value $1/(ae)$. Therefore, 
\[\log N_{k,\varepsilon}\leq c_0(k)k^{1/\varepsilon},\]with
\[c_0(k)=\frac{k-1}{   e\di\log\frac{2k}{k+1}}   +\omega_2. \] 
If \[\varepsilon\leq  \varepsilon_0(k)=\frac{\di\log\frac{k+1}{2}}{\log 8},\]
\[x_2=\left(\frac{k+1}{2}\right)^{1/\varepsilon}\geq 8.\] Lemma \ref{20250723_3} gives
\[ \pi\left(x_2\right)\leq \frac{x_2}{2}.\]
The result follows. \qedhere

\end{proof}

Notice that
 \begin{equation}\label{20250729_1}\lim_{k\to\infty}\frac{c_1(k)}{k}=\frac{1}{2e \log 2  }=0.2654.\end{equation}

In Table \ref{table1}, we list down the values  of $\varepsilon_0(k)$, $c_0(k)$ and  $c_1(k)$  for $2\leq k\leq 25$. 
\begin{table}[h]
\caption{\label{table1} The values of $\varepsilon_0(k)$, $c_0(k)$ and $c_1(k)$ for $2\leq k\leq 25$.}
\begin{tabular}{|c|c|c|c|}

\hline
$\hsp k\hsp$	&	$ \varepsilon_0(k)$	& 	$c_0(k)$	& 	$c_1(k)$	\\\hline
2	&$\hsp	0.1950\hsp	$&$	\hsp 2.2788	\hsp $&$	\hsp 1.6394\hsp	$\\\hline
3	&$	0.3333	$&$	2.8146	$&$	1.9073	$\\\hline
4	&$	0.4406	$&$	3.3482	$&$	2.1741	$\\\hline
5	&$	0.5283	$&$	3.8807	$&$	2.4403	$\\\hline
6	&$	0.6025	$&$	4.4126	$&$	2.7063	$\\\hline
7	&$	0.6667	$&$	4.9443	$&$	2.9721	$\\\hline
8	&$	0.7233	$&$	5.4757	$&$	3.2379	$\\\hline
9	&$	0.7740	$&$	6.0070	$&$	3.5035	$\\\hline
10	&$	0.8198	$&$	6.5382	$&$	3.7691	$\\\hline
11	&$	0.8617	$&$	7.0693	$&$	4.0346	$\\\hline
12	&$	0.9001	$&$	7.6003	$&$	4.3002	$\\\hline
13	&$	0.9358	$&$	8.1313	$&$	4.5657	$\\\hline
14	&$	0.9690	$&$	8.6623	$&$	4.8311	$\\\hline
15	&$	1.0000	$&$	9.1932	$&$	5.0966	$\\\hline
16	&$	1.0292	$&$	9.7241	$&$	5.3621	$\\\hline
17	&$	1.0566	$&$	10.2550	$&$	5.6275	$\\\hline
18	&$	1.0826	$&$	10.7859	$&$	5.8929	$\\\hline
19	&$	1.1073	$&$	11.3167	$&$	6.1584	$\\\hline
20	&$	1.1308	$&$	11.8476	$&$	6.4238	$\\\hline
21	&$	1.1531	$&$	12.3784	$&$	6.6892	$\\\hline
22	&$	1.1745	$&$	12.9092	$&$	6.9546	$\\\hline
23	&$	1.1950	$&$	13.4401	$&$	7.2200	$\\\hline
24	&$	1.2146	$&$	13.9709	$&$	7.4854	$\\\hline
25	&$	1.2335	$&$	14.5017	$&$	7.7509	$\\\hline

\end{tabular}
\end{table}

 For the value of $d_k(\widetilde{N}_{k,\varepsilon})$, we have the following.
\begin{theorem}
\label{20250805_4}
 Let $\varepsilon$ be a positive number, and let $\widetilde{N}_{k,\varepsilon}$ be the largest superior $k$-highly composite number  associated to $\varepsilon$ defined by \eqref{20250731_2}. For $m\in\mathbb{Z}^+$, let
$x_{m}$ be the number defined by \eqref{20250805_6}. Then
 \begin{equation}\label{20250805_5}\log d_k\left(\widetilde{N}_{k,\varepsilon}\right)=\varepsilon\sum_{m=1}^{\infty}\pi\left(x_m\right)\log x_m,\end{equation}
 where $\pi(x)$ is the  prime counting function.
 \end{theorem}
 \begin{proof}As in Theorem \ref{20250805_3}, the right hand side of \eqref{20250805_5} is a finite sum.
 Using the  same reasoning, we have
 \begin{align*}
 \log d_k\left(\widetilde{N}_{k,\varepsilon}\right)&=\sum_{p\leq k^{1/\varepsilon}}\log \binom{\widetilde{m}(p,k,\varepsilon)+k-1}{k-1}\\
 &=\sum_{m=1}^{\infty}\log\binom{m+k-1}{k-1}\sum_{x_{m+1}<p\leq x_m}1\\
 &= \sum_{m=1}^{\infty}\log\binom{m+k-1}{k-1}\Bigl(\pi\left(x_m\right)-\pi\left(x_{m+1}\right)\Bigr)\\
 &= \sum_{m=1}^{\infty}\log\binom{m+k-1}{k-1}\pi\left(x_m\right)- \sum_{m=1}^{\infty}\log\binom{m+k-2}{k-1}\pi\left(x_m\right). \end{align*}
 Since
 \[\log\binom{m+k-1}{k-1}-\log\binom{m+k-2}{k-1}=\log\frac{m+k-1}{m},\]
 we find that
 \begin{align*}
 \log d_k\left(\widetilde{N}_{k,\varepsilon}\right) &= \sum_{m=1}^{\infty} \log\frac{m+k-1}{m}\pi\left(x_m\right)\\
 &=\varepsilon\sum_{m=1}^{\infty}\pi(x_m) \log x_m.\qedhere\end{align*}

 \end{proof}
 
 Before we end this section, let us prove the following theorem which generalizes Lemma 4 in \cite{Nicolas}. It is stated as Lemma 5.3 in \cite{Duras} without proof.
\begin{theorem}\label{20250717_5}
Let $\varepsilon_0$ be a positive number and let $N_0$ be a superior $k$-highly composite number  associated to $\varepsilon_0$. If $\varepsilon_1$ and $N_1$ are positive numbers such that $N_1\geq N_0$ and 
\[\frac{d_k(n)}{n^{\varepsilon_1}}\leq \frac{d_k(N_1)}{N_1^{\varepsilon_1}}\hspace{1cm}\text{for all}\;n\geq N_0,\]
then we have the following.
\begin{enumerate}[(i)]
\item $\varepsilon_1\leq \varepsilon_0$.
\item $N_1$ is a superior $k$-highly composite number associated to $\varepsilon_1$.

\end{enumerate}
\end{theorem}
\begin{proof}
Since $N_0$ is a superior $k$-highly composite number associated to $\varepsilon_0$,
\[\frac{d_k(n)}{n^{\varepsilon_0}}\leq\frac{d_k(N_0)}{N_0^{\varepsilon_0}}\hspace{1cm}\text{for all}\;n\in\mb{Z}^+.\]In particular,
\[\frac{d_k(N_1)}{N_1^{\varepsilon_0}}\leq\frac{d_k(N_0)}{N_0^{\varepsilon_0}}.\]
By assumption, we also have
\[\frac{d_k(N_0)}{N_0^{\varepsilon_1}}\leq  \frac{d_k(N_1)}{N_1^{\varepsilon_1}}.\]
Hence,
\[  \frac{d_k(N_1)}{N_1^{\varepsilon_0}}\leq\frac{d_k(N_0)}{N_0^{\varepsilon_1}}N_0^{\varepsilon_1-\varepsilon_0}\leq \frac{d_k(N_1)}{N_1^{\varepsilon_1}}N_0^{\varepsilon_1-\varepsilon_0}.\]
This gives
\[N_0^{\varepsilon_1-\varepsilon_0}\geq N_1^{\varepsilon_1-\varepsilon_0}.\]
Since $N_0<N_1$, we must have
 $\varepsilon_1\leq\varepsilon_0$. 
It follows that for $n<N_0$,
\[\frac{d_k(n)}{n^{\varepsilon_1}}=\frac{d_k(n)}{n^{\varepsilon_0}}n^{\varepsilon_0-\varepsilon_1}\leq \frac{d_k(N_0)}{N_0^{\varepsilon_0}}N_0^{\varepsilon_0-\varepsilon_1}=\frac{d_k(N_0)}{N_0^{\varepsilon_1}}\leq\frac{d_k(N_1)}{N_1^{\varepsilon_1}}.\]
By assumption, this is also true if $n\geq N_0$. Hence, $N_1$ is a superior highly composite number associated to $\varepsilon_1$.
\qedhere
\end{proof}

 \bigskip
 \section{The Limit Superior and the Existence of Maximum Value}

  As we have discussed before, the function 
 \[f_k(n)=\frac{\log d_k(n)\log\log n}{\log k\log n}\] is bounded above. Hence, $\di \limsup_{n\to\infty}f_k(n)$ exists. 
 It was proved by Wigert \cite{Wigert} that  
\begin{equation}\label{20250802_1}\limsup_{n\to\infty}f_2(n)=\limsup_{n\to\infty}\frac{\log d(n)\log\log n}{\log 2\log n}=1.\end{equation}
Ramanujan gave a different proof in \cite{Ramanujan}. The generalization of \eqref{20250802_1} to any $k\geq 2$ is
\[\limsup_{n\to\infty}f_k(n)=\limsup_{n\to\infty}\frac{\log d_k(n)\log\log n}{\log k\log n}=1.\]
The fact that $\di\limsup_{n\to\infty}f_k(n)\leq 1$ can be inferred by a result of \cite{Heppner} (see also \cite{Duras}). We would like to give an elementary proof following the proof of \cite{Montgomery} for the $k=2$ case.

 For any $\varepsilon>0$, let $N_{k,\varepsilon}$ be a superior $k$-highly composite number associated to $\varepsilon$, and let
\[E(k,\varepsilon)=\frac{d_k(N_{k,\varepsilon})}{N_{k,\varepsilon}^{\varepsilon}}.\]
Then for any $n\in\mb{Z}^+$,
\begin{equation}\label{20250802_7}\frac{d_k(n)}{n^{\varepsilon}}\leq E(k,\varepsilon).\end{equation}
Take $n=1$ gives $E(k,\varepsilon)\geq 1$. We want to first obtain an upper bound of $E(k,\varepsilon)$ in terms of $k$ and $\varepsilon$ when $\varepsilon$ is small.

\begin{theorem}\label{20250802_2}
For $k\geq 2$, there exists a constant $B(k)$ depending only on $k$ such that
\[0\leq \log E(k,\varepsilon)\leq B(k)\varepsilon^2k^{\frac{1}{\varepsilon}}\log\frac{1}{\varepsilon} \hspace{1cm}\text{for all}\;0<\varepsilon<\frac{1}{2}.\]
\end{theorem}
\begin{proof}We follow the idea in  \cite{Montgomery} for the $k=2$ case. Since $E(k,\varepsilon)\geq 1$, $\log E(k,\varepsilon)\geq 0$.
For any  $ \varepsilon>0$,  
\[\log E(k,\varepsilon)=\log \frac{d_k( \widetilde{N}_{k,\varepsilon})}{\widetilde{N}_{k,\varepsilon}}.\]
As in the proof of Theorem \ref{20250718_1},  Let 
\[ x_1=k^{1/\varepsilon},\hspace{1cm} x_2=\left(\frac{k+1}{2}\right)^{1/\varepsilon}.\]
Then
\begin{align*}
\log E(k,\varepsilon)=\sum_{p\leq x_2} \log \frac{d_k(p^{m_p})}{p^{ \varepsilon m_p}}+
\sum_{x_2<p\leq x_1} \log \frac{d_k(p )}{p^{ \varepsilon }},
\end{align*}where $m_p=\widetilde{m}(p,k,\varepsilon)$. 
Using   the fact that $p^{\varepsilon m_p}\geq 1$ and the simple bound \eqref{20250802_3}, we have
\[\frac{d_k(p^{m_p})}{p^{\varepsilon m_p}}\leq d_k(p^{m_p})\leq d(p^{m_p})^{k-1}=(m_p+1)^{k-1}.\]  
Therefore,
\begin{equation}\label{20250802_5}\begin{split}
\log E(k,\varepsilon)& \leq (k-1)\sum_{p\leq  x_2}\log(m_p+1)+\sum_{ p\leq x_1}\left(\log k-\varepsilon\log p\right).
\end{split}
\end{equation}
Since $p^{\varepsilon}-1\geq \varepsilon\log p$, 
we find that when $0<\varepsilon<\frac{1}{2}$, 
\[m_p+1\leq\frac{k-1}{p^{\varepsilon}-1}+1\leq\frac{k-1}{\varepsilon\log p}+1\leq\frac{k-1}{\varepsilon\log 2}+1=\frac{k-1+\log 2}{\varepsilon\log 2}\leq\frac{k}{\varepsilon\log 2}\leq\frac{k}{ \varepsilon^2}.\]
Hence,
\[\log(m_p+1)\leq 2\log\frac{1}{\varepsilon}+\log k=\left(2+\frac{\log k}{\log\frac{1}{\varepsilon}}\right)\log\frac{1}{\varepsilon}\leq  b_1\log\frac{1}{\varepsilon},\]where
\[b_1=2+\frac{\log k}{\log 2}.\]
It follows that
\begin{align*}
 (k-1)\sum_{p\leq  x_2}\log(m_p+1)\leq  b_1(k-1)\left(\log\frac{1}{\varepsilon}\right) \pi(x_2).
\end{align*} By
Lemma \ref{20250723_3},
\begin{align*}
\pi(x_2)\leq\frac{\omega_1x_2}{\log x_2}&=\varepsilon\frac{\omega_1 }{\di\log\frac{k+1}{2}}\left(\frac{k+1}{2}\right)^{1/\varepsilon}=b_2\varepsilon k^{1/\varepsilon}\left(\frac{k+1}{2k}\right)^{1/\varepsilon},
\end{align*}where
\[b_2=\frac{\omega_1 }{\di\log\frac{k+1}{2}}.\]
Now,
\[\left(\frac{k+1}{2k}\right)^{1/\varepsilon}=\frac{1}{\di\exp\left(\frac{a}{\varepsilon}\right)},\]where
 \[ a=\log\frac{2k}{k+1}>0.\]
When $x>0$, $e^x>x$. 
Hence,
\[\pi(x_2)\leq \frac{b_2}{a}\varepsilon^2 k^{1/\varepsilon},\] and so
\[ (k-1)\sum_{p\leq  x_2}\log(m_p+1)\leq b_3 \varepsilon^2 k^{ 1/\varepsilon}\log\frac{1}{\varepsilon},\]
where
\[b_3=\frac{b_1b_2(k-1)}{a}=\frac{\di \omega_1(k-1)\left(2+\frac{\log k}{\log 2}\right)}{\di \log\frac{k+1}{2}\log\frac{2k}{k+1}}.\]
For the second term in \eqref{20250802_5}, we have
\begin{align*}
\sum_{ p\leq x_1}\left(\log k-\varepsilon\log p\right)=\pi(x_1)\log k-\varepsilon\vartheta(x_1)=\varepsilon\left(\pi(x_1)\log x_1-\vartheta(x_1)\right).
\end{align*}By an elementary result in analytic number theory (see for example, \cite{Apostol_1}),  there is a constant $b_4$ so that for all $x\geq 2$,
\[0\leq\pi(x)\log x-\vartheta(x)=\int_2^x\frac{\pi(t)}{  t}dt\leq \frac{b_4 x}{\log x}.\]
Now  
\[\frac{x_1}{\log x_1}=\frac{\varepsilon k^{1/\varepsilon}}{\log k}.\]
Therefore,
\begin{align*}\sum_{ p\leq x_1}\left(\log k-\varepsilon\log p\right)\leq b_5\varepsilon^2k^{\frac{1}{\varepsilon}},
\end{align*}
where
\[b_5=\frac{b_4}{\log k}.\]
Collecting the results,  we find that for all $0<\varepsilon<\frac{1}{2}$, 
\[\log E(k,\varepsilon)\leq  B(k)\varepsilon^2k^{\frac{1}{\varepsilon}}\log\frac{1}{\varepsilon}, \]
where
\[B(k)=b_3+\frac{b_5}{\log 2}\] is a constant that depends only on $k$. This completes the proof.
\end{proof}

\begin{remark}
Theorem \ref{20250805_3} and Theorem \ref{20250805_4} give
\begin{align*}\log E(k,\varepsilon)&=\varepsilon\sum_{m=1}^{\lfloor (k-1)/(2^{\varepsilon}-1)\rfloor}\left(\pi(x_m)\log x_m-\vartheta (x_m)\right)\\
& =\varepsilon\sum_{m=1}^{\lfloor (k-1)/(2^{\varepsilon}-1)\rfloor}\int_2^{x_m}\frac{\pi(t)}{t}dt,
\end{align*}where
\[x_m=\left(1+\frac{k-1}{m}\right)^{1/\varepsilon}.\]
One can then use results for the upper bounds of $\pi(x)$ to obtain a better upper bound for $\log E(k,\varepsilon)$. We do not need it for our work here.
\end{remark}
 
Now we can prove the following. 
\begin{theorem}\label{20250802_10}
For $k\geq 2$,
\[\limsup_{n\to\infty}f_k(n)=\limsup_{n\to\infty}\frac{\log d_k(n)\log\log n}{\log k\log n}=1.\]
\end{theorem}
\begin{proof}
 When $n \geq 3$, 
 \eqref{20250802_7} gives
 \[\log d_k(n)\leq \varepsilon\log n+\log E(k,\varepsilon).\]
By Theorem \ref{20250802_2}, for any $0<\varepsilon<\frac{1}{2}$, 
 \begin{equation}\label{20250802_9}\log d_k(n)\leq \varepsilon\log n+ B(k)\varepsilon^2k^{\frac{1}{\varepsilon}}\log\frac{1}{\varepsilon}.\end{equation}
 When $n>e^{k^2}$,   take
 \[\varepsilon=\frac{\log k}{\log\log n}>0.\]Then $\varepsilon<\frac{1}{2}$, and we have
 \begin{align*}
 \log d_k(n)&\leq \frac{\log k\log n}{\log\log n}+B(k)(\log k)^2\frac{\log n}{(\log\log n)^2}\log \frac{\log\log n}{\log k}.
 \end{align*}Hence,
 \[\frac{\log d_k(n)\log\log n}{\log k\log n}\leq 1+ \frac{B(k)\log k }{\log\log n}\log \frac{\log\log n}{\log k}.\]
 Taking limit superior of both sides, we find that
 \[\limsup_{n\to\infty} \frac{\log d_k(n)\log\log n}{\log k\log n}\leq 1.\]
To complete the proof, we need to show that
\[ \limsup_{n\to\infty} \frac{\log d_k(n)\log\log n}{\log k\log n}\geq 1.\]For this, it is sufficient to show that there exists a sequence $\{N_r\}$ of positive integers so that
\[\lim_{r\to\infty} \frac{\log d_k(N_r)\log\log N_r}{\log k\log N_r}=1.\]
The proof follows exactly the same idea as the $k=2$ case given in \cite{Wigert, Ramanujan}. For $r\geq 1$, let $p_r$ be the $r$-th prime. Consider the number
\[N_r=p_1p_2\cdots p_r,\]which is the product of the first $r$ primes. Then
\[\log N_r=\sum_{i=1}^r\log p_i=\sum_{p\leq p_r}\log p=\vartheta(p_r).\]
On the other hand, $\pi(p_r)=r$.
Now,
\[d_k(N_r)=k^r.\]
Hence,
\[\frac{\log d_k(N_r)\log\log N_r}{\log k\log N_r}=\frac{r\log k\log\log N_r}{\log k\log N_r}= \frac{\pi(p_r)\log\vartheta(p_r)}{\vartheta (p_r)}.\]By the prime number theorem,
\[\lim_{x\to\infty}\frac{\pi(x)\log\vartheta(x)}{\vartheta(x)}=1.\]
Hence,
\[\lim_{r\to\infty} \frac{\log d_k(N_r)\log\log N_r}{\log k\log N_r}=1.\]
 This completes the proof.
\end{proof}
\begin{remark}
The reason to take $\varepsilon=\log k/\log\log n$ in \eqref{20250802_9} is the following. The inequality  \eqref{20250802_9} is valid for all $0<\varepsilon<\frac{1}{2}$. To obtain the best result, we choose $\varepsilon$ so as to minimize the function
\[G(\varepsilon)=\varepsilon\log n+ B(k)\varepsilon^2k^{\frac{1}{\varepsilon}}\log\frac{1}{\varepsilon}.\]
Now
\[G'(\varepsilon)=\log n+B(k)\left(2\varepsilon k^{\frac{1}{\varepsilon}}\log\frac{1}{\varepsilon}-  k^{\frac{1}{\varepsilon}}\log k\log\frac{1}{\varepsilon}
-\varepsilon k^{\frac{1}{\varepsilon}}\right).\]When $n$ is large, we find that $G'(\varepsilon)$ is close to zero when 
\[\varepsilon=\frac{\log k}{\log \log n}.\]When $\varepsilon$ goes to 0, $G'(\varepsilon)<0$. When $\varepsilon\to \infty$, $G'(\varepsilon)>0$. Therefore, when $n$ is large, the minimum of $G(\varepsilon)$ indeed appears in the vicinity of $\log k/\log\log n$. 
\end{remark}

Now to show that $f_k(n)$ has a maximum value, we need the following. 
\begin{lemma}\label{20250731_12}
Let $N_0=2520=2^3\times 3^2\times 5\times 7$. For $k\geq 2$, let
\[a_k=\frac{\log d_k(N_0)}{\log k}.\]
Then the sequence $\{a_k\}$ is increasing. 
 
\end{lemma}
\begin{proof}
Note that
\[a_k=\frac{\di \log\frac{k^4(k+1)^2(k+2)}{12}}{\log k}=4+ \frac{2\log(k+1)}{\log k}+\frac{\log (k+2)}{\log k}-\frac{\log 12}{\log k}=4+G_1(k),\]where
\[G_1(t)=\frac{2\log(t+1)}{\log t}+\frac{\log (t+2)}{\log t}-\frac{\log 12}{\log t},\quad t>1.\]
Note that
\begin{align*}
G_1'(t)= \frac{1}{t\log^2 t}\left(\log 12+2\left[\frac{t\log t}{t+1}-\log(t+1)\right]+ \frac{t\log t}{t+2}-\log(t+2) \right).
\end{align*}For any positive number $a$, let
\[G_2(t)=\frac{t\log t}{t+a}-\log(t+a),\quad t>1.\]
Then
\[G_2'(t)= \frac{a\log t}{(t+a)^2}>0.\]
Hence, the function $G_2(t)$ is increasing on $(1,\infty)$. This implies that for $t\geq 2$,
\begin{align*}
&\log 12+2\left[\frac{t\log t}{t+1}-\log(t+1)\right]+ \frac{t\log t}{t+2}-\log(t+2)\\
&\geq \log 12+2\left(\frac{2\log 2}{3}-\log 3\right)+ \frac{2\log 2}{4}-\log 4=0.1722>0.
\end{align*}Hence, when $t\geq 2$, $G_1'(t)>0$. This implies that for any $k\geq 2$, $G(k+1)\geq G(k)$. Hence, the sequence $\{a_k\}$ is increasing.
\end{proof}

 Lemma \ref{20250731_12} gives a lower bound to the value of $\di\lambda(k)=\sup_{n\geq 3}f_k(n)$. The choice of $N_0=2520$ is just for convenience.
\begin{corollary}\label{20250731_14}
Let $N_0=2520$. For any $k\geq 2$, 
\[f_k(N_0)=\frac{\log d_k(N_0)\log\log N_0}{\log k\log  N_0}\geq 1.4677.\]
Thus,
\[\sup_{n \geq 3} f_k(n) \geq 1.4677.\]
\end{corollary}
\begin{proof}For $N_0=2520$, $d_2(N_0)=48$.
It is a straightforward computation to get
\[f_2(N_0)=\frac{\log 48\log\log 2520}{\log 2\log 2520}=1.4677.\]
By Lemma \ref{20250731_13},
\[\frac{\log d_k(N_0)}{\log k}\geq \frac{\log d_2(N_0)}{\log 2}.\]
Therefore,
\[ f_k(N_0)=\frac{\log d_k(N_0)\log\log N_0}{\log k\log N_0}\geq \frac{\log d_2(N_0)\log\log N_0}{\log 2\log N_0} =f_2(N_0)=1.4677.\]
\end{proof}

Now we can show that the function $f_k(n)$ indeed has a maximum value.
\begin{theorem}\label{20250803_2}
For $k\geq 2$, the function
\[f_k(n)=\frac{\log d_k(n)\log\log n}{\log k\log n},\quad n \geq 3\] has a maximum value.
\end{theorem}
\begin{proof}
By Corollary \ref{20250731_14}, 
\[\sup_{n \geq 3} f_k(n)\geq 1.4677.\] By Theorem \ref{20250802_10}, there is a positive integer $ N_1$ such that for all $n>N_1$,
\[f_k(n)<1.4.\]
Therefore,
\[\sup_{n \geq 3}f_k(n)=\sup_{2\leq n\leq N_1}f_k(n).\]
Since there are only finitely many integers between 2 and $N_1$, one of them will give the maximum value of $f_k(n)$. 
\end{proof}

\bigskip
\section{Finding the Maximizer}
 In this section, we want to establish an effective algorithm to find
\[\lambda(k)=\max_{n \geq 3}f_k(n)= \max_{n \geq 3}\frac{\log d_k(n)\log\log n}{\log k\log n},\] and an integer $N_{\max}(k)$ such that
\[\lambda(k)=f_k(N_{\max}(k)).\]
We use the same strategy as in \cite{Nicolas}. We will show that the maximizer of $f_k(n)$ must be a superior $k$-highly composite number  associated to some $\varepsilon>0$. We then narrow down the range of $\varepsilon$ to some interval $[\varepsilon_1(k), \varepsilon_2(k)]$, which leaves us with finitely many candidates for the superior $k$-highly composite numbers  to be considered. One can then use a computer algorithm to list down all these candidates and find the maximum value and maximizer of $f_k(n)$ from there.

We start with the following theorem that gives a sufficient condition for a  lower bound  of the maximizer of $f_k(n)$.
\begin{theorem}\label{20250717_3}
Assume that $N$ is  a superior $k$-highly composite number   associated to some positive number $\varepsilon$. If $N\geq 3$ and
\begin{equation}\label{20250723_4}\log d_k(N )- \varepsilon\log N < e^2\varepsilon,\end{equation}
then for $2\leq n< N $, $f_k(n)< f_k(N )$. 
\end{theorem}
\begin{proof}
Since $N$ is a superior highly composite number associated to $\varepsilon$, for all $n\geq 1$,
\begin{equation}\label{20250723_5}\frac{d_k(n)}{n^{ \varepsilon }}\leq\frac{d_k(N)}{N^{ \varepsilon }}.\end{equation}
Hence, 
\begin{equation}\label{20250731_3}\log d_k(n)\leq  \varepsilon  \log n+a_1,\end{equation}where
\[a_1=\log d_k(N)- \varepsilon \log N.\]
Taking $n=1$ in \eqref{20250731_3}, we find that $a_1\geq 0$.
For $n \geq 3$,
\begin{align*}
f_k(n)=\frac{\log d_k(n) \log\log n}{\log k \log n}\leq \frac{1}{\log k}\left( \varepsilon \log\log n +\frac{a_1\log\log n }{\log n}\right)=\frac{h(\log\log n)}{\log k},
\end{align*}
where
\[h(t)= \varepsilon t+a_1te^{-t}.\]
We have
\[h'(t)=  \varepsilon-a_1(t-1)e^{-t},\hspace{1cm}h''(t)= a_1(t-2)e^{-t}.\]Therefore, when $t<2$, $h''(t)\leq 0$; when $t>2$, $h''(t)\geq 0$. This implies that $h'(t)$ is decreasing on $(-\infty, 2]$ and increasing on $[2,\infty)$. Hence, $h'(t)$ has a minimum value at $t=2$, with minimum value 
\[h'(2)=  \varepsilon -a_1e^{-2}=e^{-1}\left(\varepsilon e^2 -\log d_k(N)+ \varepsilon \log N\right).\]  
The assumption \eqref{20250723_4} implies that $h'(2)>0$.
 Thus, $h'(t)> 0$ for all $t$, and so $h(t)$ is a strictly increasing function. Therefore,
for $2\leq n< N$, 
\begin{align*}
f_k(n)&<\frac{h(\log\log N)}{\log k}= \frac{1}{\log k}\left( \varepsilon \log\log N+\frac{a_1\log\log N}{\log N } \right)\\
&=\frac{ \log d_k(N )\log\log N }{ \log k\log N }=f_k(N). \end{align*}
\end{proof}

The following theorem proves the existence of a positive number $\varepsilon$ that satisfies the condition in Theorem \ref{20250717_3}.

\begin{theorem}\label{20250719_1}
There exists a positive number $\varepsilon_2(k)$  and  a superior $k$-highly composite number $N_2(k)$  associated to $\varepsilon_2(k)$, such that   $N_2(k)>e^{e^2}=1618.2$ and
\begin{equation}\label{20250731_7}0\leq \log d_k(N_2(k))- \varepsilon_2(k)\log N_2(k)\leq e^2\varepsilon_2(k).\end{equation}
Any positive integer less than $N_2(k)$ cannot be a maximizer of the function $f_k(n)$.
\end{theorem}
\begin{proof} By Theorem \ref{20250717_3}, if we can find $N_2(k)$ satisfying \eqref{20250731_7}, then for all $2\leq n<N_2(k)$, 
\[f_k(n)<f_k(N_2(k)).\]
This implies that  if $2\leq n< N_2(k)$, $n$ is not a maximizer of the function $f_k$. To prove the theorem, we need to find $\varepsilon_2(k)$ and $N_2(k)$.
 We claim that 
 \begin{equation}\label{20250731_13}\varepsilon_2(k)=\begin{cases} \di\frac{\log k}{2},\quad & 2\leq k\leq 25,\\[2ex]\di\frac{\log\di\left( \frac{k+10}{11}\right)}{\log 2},\quad &  k\geq 26,\end{cases}\end{equation} and
  \begin{equation}\label{20250731_6}N_2(k)=\widetilde{N}_{k,\varepsilon_2(k)}=\prod_{p\leq k^{1/\varepsilon_2(k)}}p^{\widetilde{m}(p,k,\varepsilon_2(k)},\end{equation}
  with
  \[\widetilde{m}(p,k,\varepsilon_2(k)=\left\lfloor\frac{k-1}{p^{\varepsilon_2(k)}-1}\right\rfloor\]satisfies the conditions of the theorem.

  By Theorem \ref{20250717_2}, $N_2(k)$ defined by \eqref{20250731_6} is indeed a superior $k$-highly composite number associated to $\varepsilon_2(k)$. It follows from definition of superior $k$-highly composite numbers that
\[\frac{d_k(N_2(k))}{N_2(k)^{\varepsilon_2(k)}}\geq\frac{d_k(1)}{1^{\varepsilon_2(k)}}=1.\]
Hence, 
  \[ \log d_k(N_2(k))- \varepsilon_2(k)\log N_2(k)\geq 0.\]
 When $k\geq 26$,  $\varepsilon_2(k)$ is defined so that
\[2^{\varepsilon_2(k)}-1=\frac{k-1}{11},\]
and so
\[\widetilde{m}(2,k,\varepsilon_2(k))=11.\]
It follows that 
\[N_2(k)\geq 2^{\widetilde{m}(2,k,\varepsilon_2(k))}=2^{11}=2048>e^{e^2}.\]
  Now we show   analytically that when $k\geq 636$, $\varepsilon_2(k)$ and $N_2(k)$ indeed satisfy     
\[0\leq \log d_k(N_2(k))- \varepsilon_2(k)\log N_2(k)\leq e^2\varepsilon_2(k).\]
Let
\[t=\varepsilon_2(k)=\frac{\di \log \frac{k+10}{11}}{\log 2}, \quad a=\log 2, \quad b=\log 3,\quad c=\log 11.\]
Then
\[\frac{k+10}{11}=e^{at},\]
\[\frac{k-1}{3^{\varepsilon_2(k)}-1}=\frac{k+10-11}{e^{b t}-1}=\frac{11(e^{at }-1)}{e^{bt}-1}=11h(t),\]
where
\[h(t)=\frac{e^{at }-1}{e^{bt}-1}.\]
Now,
\begin{align*}
h'(t)=-\frac{(b-a)e^{(a+b)t}+ae^{at}-be^{bt}}{(e^{bt}-1)^2}.
\end{align*}When $k\geq 636$, 
\[t=\varepsilon_2(k)\geq 5.87,\]
\[e^{at}=\frac{646}{11}>2.71=\frac{\log 3}{\log 3-\log 2}=\frac{b}{b-a}.\]Thus,
\[(b-a)e^{at}>b\] and so 
\[(b-a)e^{(a+b)t}+ae^{at}-be^{bt}\geq e^{bt}\left( (b-a)e^{at}-b\right)>0.\]
Hence, when $k\geq 636$, $h'(t)<0$, and so $h(t)$ is a decreasing function. This implies that when $k\geq 636$,
\begin{equation}\label{20250729_21}\frac{k-1}{3^{\varepsilon_2(k)}-1}\leq\frac{635}{3^{\varepsilon_2(636)}-1}=0.9998<1.\end{equation}
In fact, $k=636$ is the smallest positive integer such that this inequality holds. 
It follows that for  all $p\geq 3$,
\[\frac{k-1}{p^{\varepsilon_2(k)}-1}<1.\]
Thus, for all $p\geq 3$,
\[\widetilde{m}(p,k,\varepsilon_2(k))<1.\]
This implies that
\[N_2(k)=2^{\widetilde{m}(2,k,\varepsilon_2(k))}=2^{11}=2048,\]
and so
\[\log d_k(N_2(k))=\log\binom{k+10}{11}.\]
It follows that
\begin{align*}
  \log d_k(N_2(k))- \varepsilon_2(k)\log N_2(k)&=\log\binom{k+10}{11}-11\log\frac{k+10}{11}\\
&=\log \frac{k(k+1)\cdots (k+10)}{(k+10)^{11}}+\log\frac{11^{11}}{11!}
\\&<\log\frac{11^{11}}{11!}<8.88,
\end{align*}while
\[e^2\varepsilon_2(k)>43.41\]
Hence, for all $k\geq 636$, we have
\[0\leq \log d_k(N_2(k))- \varepsilon_2(k)\log N_2(k)\leq e^2\varepsilon_2(k).\]
For $2\leq k\leq 635$, the assertions can be verified by computing $\varepsilon_2(k)$,  $N_2(k)$ and \begin{equation}\label{20251026_1}  u(k)=\log d_k(N_2(k))- \varepsilon_2(k)\log N_2(k)-e^2\varepsilon_2(k)\end{equation} explicitly. The results for $2\leq k\leq 25$ and $k=635, 636$ are tabled in Table \ref{table2}. This completes the proof.
\end{proof}
\begin{remark}
Theorem \ref{20250719_1} is technical.  We want $N_2(k)> 1618$. A natural candidate for $N_2(k)$ is $\widetilde{N}_{k,\varepsilon_2(k)}$.  If $\widetilde{N}_{k,\varepsilon_2(k)}$ contains $2^{11}$ as a factor, we can guarantee $\widetilde{N}_{k,\varepsilon_2(k)}>1618$. When  $\varepsilon_2(k)=\di\log_2 \left( \frac{k+10}{11}\right)$, $\widetilde{m}(2,k,\varepsilon_2(k))=11$ and so $2^{11}$ is a factor of $\widetilde{N}_{k,\varepsilon_2(k)}$. When $k\geq 636$, we can show that  $\widetilde{N}_{k,\varepsilon_2(k)}=2^{11}$. Using this, we can prove that  \eqref{20250731_7} holds. There are only finitely many  $k$ which is less than 636. It is easier  to check numerically for these cases. We find that when $2\leq k\leq 19$, $\varepsilon_2(k)=\di \log_2\left( \frac{k+10}{11}\right) $ and $N_2(k)=\widetilde{N}_{k,\varepsilon_2(k)}$ do not satisfy \eqref{20250731_7}. After some trial and error, we find (numerically) that $\varepsilon_2(k)=\frac{1}{2}\log k$ and $N_2(k)=\widetilde{N}_{k,\varepsilon_2(k)}$ satisfy $N_2(k)>1618$ and    \eqref{20250731_7}  for $2\leq k\leq 29$.  
\end{remark}

\begin{table}[h]
\caption{\label{table2} The values of $\varepsilon_0(k)$, $\varepsilon_2(k)$, $N_2(k)$ and $u(k)$ (see  \eqref{20251026_1}) for $2\leq k\leq 25$ and $k=635, 636$.}
\begin{tabular}{|c|c|c|c| c|}

\hline
$\hsp k \hsp$	&	$ \varepsilon_0(k) $ &	$ \varepsilon_2(k) $     & $  N_2(k) $ & $  u(k)$\\\hline
2	&$\hsp	0.1950\hsp	$ &$	\hsp 0.3466\hsp    $&$\hsp	2520=2^{3}\times 3^{2}\times 5 \times 7 	\hsp $&$\hsp	-1.4040\hsp	$\\\hline
 
3	&$	0.3333	$&$	0.5493	$&$	5040=2^{4}\times 3^2\times 5\times 7	$&$	-2.0447	$\\\hline
4	&$	0.4406	$&$	0.6931	$&$	5040=2^{4}\times 3^2\times 5\times 7	$&$	-2.4004	$\\\hline
5	&$	0.5283	$&$	0.8047	$&$	10080=2^{5}\times 3^2\times 5\times 7	$&$	-2.6011	$\\\hline
6	&$	0.6025	$&$	0.8959	$&$	10080=2^{5}\times 3^2\times 5\times 7	$&$	-2.7207	$\\\hline
7	&$	0.6667	$&$	0.9730	$&$	60480=2^{6}\times 3^3\times 5\times 7	$&$	-2.7502	$\\\hline
8	&$	0.7233	$&$	1.0397	$&$	60480=2^{6}\times 3^3\times 5\times 7	$&$	-2.7358	$\\\hline
9	&$	0.7740	$&$	1.0986	$&$	120960=2^{7}\times 3^3\times 5\times 7	$&$	-2.7051	$\\\hline
10	&$	0.8198	$&$	1.1513	$&$	120960=2^{7}\times 3^3\times 5\times 7	$&$	-2.6371	$\\\hline
11	&$	0.8617	$&$	1.1989	$&$	120960=2^{7}\times 3^3\times 5\times 7	$&$	-2.5634	$\\\hline
12	&$	0.9001	$&$	1.2425	$&$	241920=2^{8}\times 3^3\times 5\times 7	$&$	-2.4825	$\\\hline
13	&$	0.9358	$&$	1.2825	$&$	241920=2^{8}\times 3^3\times 5\times 7	$&$	-2.3803	$\\\hline
14	&$	0.9690	$&$	1.3195	$&$	241920=2^{8}\times 3^3\times 5\times 7	$&$	-2.2780	$\\\hline
15	&$	1.0000	$&$	1.3540	$&$	725760=2^{8}\times 3^4\times 5\times 7	$&$	-2.1599	$\\\hline
16	&$	1.0292	$&$	1.3863	$&$	1451520=2^{9}\times 3^4\times 5\times 7	$&$	-2.0210	$\\\hline
17	&$	1.0566	$&$	1.4166	$&$	1451520=2^{9}\times 3^4\times 5\times 7	$&$	-1.8843	$\\\hline
18	&$	1.0826	$&$	1.4452	$&$	1451520=2^{9}\times 3^4\times 5\times 7	$&$	-1.7505	$\\\hline
19	&$	1.1073	$&$	1.4722	$&$	2903040=2^{10}\times 3^4\times 5\times 7	$&$	-1.6104	$\\\hline
20	&$	1.1308	$&$	1.4979	$&$	2903040=2^{10}\times 3^4\times 5\times 7	$&$	-1.4650	$\\\hline
21	&$	1.1531	$&$	1.5223	$&$	2903040=2^{10}\times 3^4\times 5\times 7	$&$	-1.3229	$\\\hline
22	&$	1.1745	$&$	1.5455	$&$	2903040=2^{10}\times 3^4\times 5\times 7	$&$	-1.1841	$\\\hline
23	&$	1.1950	$&$	1.5677	$&$	5806080=2^{11}\times 3^4\times 5\times 7	$&$	-1.0365	$\\\hline
24	&$	1.2146	$&$	1.5890	$&$	5806080=2^{11}\times 3^4\times 5\times 7	$&$	-0.8888	$\\\hline
25	&$	1.2335	$&$	1.6094	$&$	5806080=2^{11}\times 3^4\times 5\times 7	$&$	-0.7444	$\\\hline
635	&$	2.7710	$&$	5.8737	$&$	6144=2^{11}\times 3	$&$	-34.6118	$\\\hline
636	&$	2.7717	$&$	5.8760	$&$	2048=2^{11}	$&$	-34.6289	$\\\hline

\end{tabular}
\end{table}

The next lemma  gives a comparison between $\varepsilon_2(k)$ defined by \eqref{20250731_13} to  $\varepsilon_0(k)$   defined by \eqref{20250720_6}.
\begin{lemma}\label{20250809_1}
For $k\geq 2$, let $\varepsilon_0(k)$ be the number defined by \eqref{20250720_6}, and let $\varepsilon_2(k)$ be the number defined by \eqref{20250731_13}. Then $\varepsilon_0(k)\leq\varepsilon_2(k)$.

\end{lemma}
\begin{proof}
For $2\leq k\leq 25$, one can check directly that $\varepsilon_0(k)\leq\varepsilon_2(k)$ from the values of $\varepsilon_0(k)$ and $\varepsilon_2(k)$ listed  in  Table \ref{table2}. For  $k\geq 26$, we consider the polynomial function
\[g(u)=2(u+10)^3-11^3(u+1).\]
We find that
\[g'(u)=6(u+10)^2-1331.\]
Therefore, $g'(u)>0$ when 
\[u>-10+\sqrt{\frac{1331}{6}}=4.89.\]
This implies that for all $u\geq 26$, \[g(u)\geq g(26)=57375>0.\]
Hence,
for $k\geq 26$,
\[\left(\frac{k+10}{11}\right)^3>\frac{k+1}{2}.\]It follows that
\[3\log  \frac{k+10}{11}>\log\frac{k+1}{2}.\]
Therefore,
\[\varepsilon_2(k)=\frac{\di \log  \frac{k+10}{11}}{\log 2}>\frac{\di \log\frac{k+1}{2}}{3\log 2}=\frac{\di \log\frac{k+1}{2}}{ \log 8}=\varepsilon_0(k).\qedhere\]
\end{proof}

 Now we want to find an upper bound of the maximizer of $f_k(n)$.
The next   lemma is Lemma 6 in \cite{Nicolas}.  
\begin{lemma}\label{20250717_7}
Let $b$ be a real number greater than 2, and let $F_b(t)$ be the function defined as
\[F_b(t)=\frac{e^t}{t}\left(1-\frac{b-1}{b^2} t\right).\]
Then for $t\geq 2$,
\[F_b(t)\leq F_b(b)=\frac{e^b}{b^2}.\]
\end{lemma}
 
We also need the following technical lemma.
\begin{lemma}\label{20250720_8}
For $k\geq 2$, let $\varepsilon_1(k)$ be a positive number less than or equal to the number $\varepsilon_0(k)$ given by \eqref{20250720_6}, and let $c_1(k)$ be the number given by \eqref{20250720_7}. If $c_2(k)=\log c_1(k)$, then
\[\frac{\log k}{\varepsilon_1(k)}\geq \frac{c_2(k)-c_2(k)^2}{c_2(k)+1}.\]
\end{lemma}
\begin{proof}By definition, $c_1(k)>1$ and so $c_2(k)>0$. If $c_2(k)\geq 1$,
\[\frac{c_2(k)-c_2(k)^2}{c_2(k)+1}\leq 0< \frac{\log k}{\varepsilon_1(k)}.\]
If $0<c_2(k)<1$, then 
\[\frac{c_2(k)-c_2(k)^2}{c_2(k)+1}\leq \frac{1}{4},\]but
\[\frac{\log k}{\varepsilon_1(k)}\geq\frac{\log k}{\varepsilon_0(k)}=\frac{(\log 8)(\log k)}{\log \di\frac{k+1}{2}}\geq \log 8>1.\]
These prove the assertion.\qedhere
\end{proof}

 The following theorem is essentially Theorem 1.3 in \cite{Duras} in disguise. It generalizes Lemma 7 in \cite{Nicolas}.
\begin{theorem}\label{20250718_2}
Assume that the function $f_k(n)$ achieves the  maximum value $\lambda(k) $ at the positive integer $N(k)$. Then $N(k)$ is a superior $k$-highly composite number associated to 
\begin{equation}\label{20250724_1}\varepsilon(k)=\lambda(k) \log k\frac{\log\log N(k)-1}{(\log\log N(k))^2}.\end{equation}
\end{theorem}
\begin{proof}

By Theorem  \ref{20250719_1}, there exists $\varepsilon_2(k)>0$, and a superior $k$-highly composite number $N_2(k)$ associated to $\varepsilon_2(k)$ such that 
 $ N_2(k)>e^{e^2}$, and any positive integer less than $N_2(k)$ cannot be a maximizer of $f_k(n)$. In other words, $N(k)\geq N_2(k)$. 
 With $\varepsilon(k)$ given by \eqref{20250724_1}, if we can show that for all $n\geq N_2(k)$, 
\[\frac{d_k(n)}{n^{\varepsilon(k)}}\leq\frac{ d_k(N(k))}{N(k)^{\varepsilon(k)}},\]
  Theorem \ref{20250717_5}  will imply that $N(k)$ is a superior $k$-highly composite number associated to $\varepsilon(k)$. This will complete the proof of Theorem \ref{20250718_2}.
  By assumption,
\begin{equation}\label{20250717_6}\lambda(k)=f_k(N(k))=\frac{\log d_k(N(k))\log\log N(k)}{\log k \log N(k)}\geq\frac{\log d_k(n)\log\log n}{\log k \log n}\quad\text{for all}\;n \geq 3.\end{equation}
Now,
\[\log d_k(n)-\varepsilon(k)\log n=T_1+T_2,\]where
\[T_1=\log d_k(n)-\lambda(k) \frac{ \log k\log n}{\log\log n},\hspace{1cm}T_2=\lambda(k) \frac{\log k \log n}{\log\log n}-\varepsilon(k)\log n.\]
 When $n \geq 3$, by \eqref{20250717_6}, 
\[T_1=\log d_k(n)-\lambda(k)\frac{ \log k\log n}{\log\log n}\leq 0.\]
For $T_2$, we have
\[T_2=\lambda(k)(\log k) h(\log\log n),\] where
\begin{align*}
h(t)=\frac{e^t}{t}-\frac{\varepsilon(k)}{\lambda(k)\log k} e^t.
\end{align*}
From \eqref{20250724_1},
\[\frac{\varepsilon(k)}{\lambda(k)\log k}=\frac{\log\log N(k)-1}{(\log\log N(k))^2}=\frac{b-1}{b^2},\]
with $b=\log\log N(k)$. Since $N(k)\geq N_2(k)$ and $N_2(k)>e^{e^2}$, we have $b>2$. By Lemma \ref{20250717_7}, for  $t\geq 2$, 
\[h(t)\leq h(b).\]This implies that for  $n\geq e^{e^2}$, 
\begin{align*}
T_2& \leq \lambda(k) (\log k) h(\log\log N(k))\\
&= \lambda(k) \log k\frac{ \log N(k)}{\log\log N(k)}-\varepsilon(k)\log N(k)\\&=\log d_k(N(k))-\varepsilon(k)\log N(k).
\end{align*}
Therefore, when $n\geq e^{e^2}$,
\[\log d_k(n)-\varepsilon(k)\log n=T_1+T_2\leq \log d_k(N(k))-\varepsilon(k)\log N(k).\]This implies that for all $n\geq N_2(k)$, 
\[\frac{d(n)}{n^{\varepsilon(k)}}\leq\frac{d(N(k))}{N(k)^{\varepsilon(k)}},\]
which completes the proof of Theorem \ref{20250718_2}. \qedhere
\end{proof}

Theorem \ref{20250718_2} says that a maximizer of the function $f_k(n)$ must be a  superior $k$-highly composite number.  
The next theorem gives an upper bound $\lambda_1(k)$ to $f_k(N)$ when $N$ is a superior  $k$-highly composite number associated to $\varepsilon$ that is less than or equal to the number  $\varepsilon_0(k)$ defined by \eqref{20250720_6}. It generalizes Lemma 8 in \cite{Nicolas}.

\begin{theorem}\label{20250718_3}
Let $c_1(k)$ be the number given by \eqref{20250720_7}, and let $\varepsilon_1(k)$ be a number less than or equal to the number $\varepsilon_0(k)$ defined by \eqref{20250720_6}. 
If $0<\varepsilon\leq \varepsilon_1(k)$ and $N$ is a superior $k$-highly composite number  associated to $\varepsilon$, then
\[\frac{ \varepsilon (\log\log N)^2}{\log k(\log\log N-1)}\leq \lambda_1(k)\leq\lambda_0(k),\]
where
\begin{equation}\label{20250806_1}\lambda_0(k)=\frac{  (\log k+\varepsilon_0( k)\log c_1(k))^2}{( \log k )(\log k+[\log c_1(k)-1]\varepsilon_0( k))},\end{equation}
\begin{equation}\label{20250731_8}\lambda_1(k)=\frac{  (\log k+\varepsilon_1( k)\log c_1(k))^2}{( \log k )(\log k+[\log c_1(k)-1]\varepsilon_1( k))}.\end{equation}

\end{theorem}
\begin{proof}
Let $x_0=k^{1/\varepsilon_0(k)}$,  $x_1=k^{1/\varepsilon_1(k)}$ and $x=k^{1/\varepsilon}$. Then $x_0\leq x_1\leq x$,
\[\log x_0=\frac{\log k}{\varepsilon_0(k)},\hspace{1cm}\log x_1=\frac{\log k}{\varepsilon_1(k)}\hspace{1cm}\text{and}\hspace{1cm}\varepsilon=\di\frac{\log k}{\log x}.\] 
By Theorem \ref{20250718_1}, 
\begin{equation}\label{20250720_5}\log\log N \leq \log x+c,\end{equation}
where\[c=\log c_1(k)>0.\]
Consider the function
\[h_1(t)=\frac{t^2}{t-1},\quad t>1.\]
Since 
\[h_1'(t)= \frac{t(t-2)}{(t-1)^2},\]
we find that $h_1'(t)>0$ when $t>2$.  Hence,
$h_1(t)$ is increasing when $t\geq 2$.   
Let $\varepsilon_2(k)$ and $N_2(k)$ be numbers satisfying the conditions in Theorem \ref{20250719_1}. Then
\[N\geq N_2(k)>e^{e^2}.\]
It follows that $\log\log N >2$. 
Eq. \eqref{20250720_5}  then implies that
\[\frac{\varepsilon(\log\log N )^2}{\log k(\log\log N -1)}= \frac{(\log\log N )^2}{(\log x)(\log\log N -1)}\leq  \frac{(\log x+c)^2}{\log x(\log x+  c-1)}= h_2(\log x),\]
where
\[h_2(t)=\frac{(t+c)^2}{t(t+c-1)}.\]
A straightforward computation gives
\begin{align*}
h_2'(t)=-\frac{(t+c)\left([c+1]t-[c-c^2]\right)}{t^2(t+c-1)^2}.
\end{align*}
This implies that when 
\[t>\frac{c-c^2}{c+1},\]
$h_2(t)$ is decreasing. 
By Lemma \ref{20250720_8},
\[\log x_0=\frac{\log k}{\varepsilon_0(k)}>\frac{c-c^2}{c+1}.\]

Since $  x\geq x_1\geq x_0$, we have
\[  \frac{\varepsilon(\log\log N )^2}{\log k(\log\log N -1)}\leq  h_2(\log x)\leq  h_2(\log x_1)\leq  h_2(\log x_0).\]
This shows that
\[ \frac{\varepsilon(\log\log N )^2}{\log k(\log\log N -1)}\leq \lambda_1(k)\leq\lambda_0(k),\]
where
\[\lambda_0(k)=h_2(\log x_0)=\frac{ (\log x_0+c)^2}{(\log x_0)(\log x_0+c-1)}=\frac{  (\log k+\varepsilon_0( k)\log c_1(k))^2}{ (\log k)(\log k+[\log c_1(k)-1]\varepsilon_0( k))},\]
\[\lambda_1(k)=h_2(\log x_1) =\frac{  (\log k+\varepsilon_1( k)\log c_1(k))^2}{ (\log k)(\log k+[\log c_1(k)-1]\varepsilon_1( k))} .\]
\end{proof}

As $k\to\infty$,
\[\log c_0(k)\sim \log k,\quad \varepsilon_0(k)\sim \frac{\log k}{\log 8}.\]
Therefore,
\[\lambda_0(k)\sim \frac{\log k}{\log 8}\hspace{1cm}\text{as}\;\;k\to\infty.\]
 
In Table \ref{table3}, we tabulate the values of  $\lambda_0(k)$ for  $2\leq k\leq 25$. 

\renewcommand{\hsp}{\hspace{0.8cm}}
\begin{table}[h]
\caption{\label{table3} The values of $\lambda_0(k)$ for $2\leq k\leq 25$.}
\begin{tabular}{|c|c||c|c|}

\hline
$\hsp k \hsp$	&	$ \lambda_0(k) $    	&	$\hsp k\hsp $	& 	$\lambda_0(k)$ \\\hline
2	&	$\hsp 1.5126\hsp$	&	14	&	$\hsp 2.0568\hsp$	\\\hline
3	&	1.6025	&	15	&	2.0813	\\\hline
4	&	1.6735	&	16	&	2.1046	\\\hline
5	&	1.7328	&	17	&	2.1267	\\\hline
6	&	1.7842	&	18	&	2.1477	\\\hline
7	&	1.8297	&	19	&	2.1679	\\\hline
8	&	1.8706	&	20	&	2.1871	\\\hline
9	&	1.9078	&	21	&	2.2056	\\\hline
10	&	1.9420	&	22	&	2.2233	\\\hline
11	&	1.9737	&	23	&	2.2404	\\\hline
12	&	2.0032	&	24	&	2.2568	\\\hline
13	&	2.0308	&	25	&	2.2727	\\\hline

\end{tabular}
\end{table}

For fixed $k$, as  $\varepsilon_1(k)\to 0^+$,  $\lambda_1(k)\to 1$. 

\begin{remark}
One might be tempting to give an upper bound to $f_k(N_{k,\varepsilon})$ when $\varepsilon\leq\varepsilon_1(k)$. This is more tedious and it does not offer a bound that is better than $\lambda_1(k)$.
\end{remark}

Now we can conclude   the following.
\begin{theorem}\label{20250803_3}
For $k\geq 2$, the function
\[f_k(n)=\frac{\log d_k(n)\log\log n}{\log k\log n},\quad n \geq 3\] has a maximum value $\lambda(k)$ at a superior $k$-highly composite number $N_{\max}(k)$  associated to a positive number $\varepsilon(k)$. Let $\varepsilon_1(k)$ be a positive number less than or equal to  the number $\varepsilon_0(k)$ \eqref{20250720_6}, and let $\varepsilon_2(k)$ be a positive number that satisfies the conditions in Theorem \ref{20250719_1}. If the number $\lambda_1(k)$ given by \eqref{20250731_8} is less than $\lambda(k)$,  then we must have $\varepsilon_1(k)\leq\varepsilon(k)\leq\varepsilon_2(k)$. 
\end{theorem}
\begin{proof}
In Theorem \ref{20250803_2}, we have shown that $f_k(n)$ has a maximum value at some positive integer $N_{\max}(k)$. By Theorem \ref{20250718_2}, $N_{\max}(k)$ must be a  superior $k$-highly composite number associated to some $\varepsilon(k)>0$. By Theorem \ref{20250719_1}, $\varepsilon(k)\leq \varepsilon_2(k)$. By Theorem \ref{20250718_3}, if $\varepsilon(k)<\varepsilon_1(k)$, then $\lambda(k)\leq \lambda_1(k)$. This contradicts to $\lambda_1(k)$ is less than $\lambda(k)$. Therefore, we must have $\varepsilon(k)\geq\varepsilon_1(k)$. 
\end{proof}

\bigskip
\section{Numerical Algorithm and Results}
Now we can use Theorem \ref{20250803_3} to develop an effective
  algorithm to determine  
\[\lambda(k)=\max_{n \geq 3} f_k(n)=\max_{n \geq 3}\frac{\log d_k(n)\log\log n}{\log k\log N}\] and find a  positive integer  $N_{\max}(k)$ such that
\[\lambda(k)=f(N_{\max}(k)).\]

We already know that $N_{\max}(k)$ must be a superior $k$-highly composite number associated to  some $\varepsilon(k)>0$. We can fix $\varepsilon_2(k)$ to be the number defined by \eqref{20250731_13} in the proof of Theorem  \ref{20250719_1}.
 We need to find $\varepsilon_1(k)$ so that $\varepsilon_1(k)\leq\varepsilon_0(k)$ and $\varepsilon(k)\geq\varepsilon_1(k)$.

We start with  $\varepsilon_1(k)=\varepsilon_0(k)$ and compute $\lambda_1(k)$.
 Consider those superior $k$-highly composite numbers associated to $\varepsilon$ with $\di \varepsilon_1(k)\leq \varepsilon\leq \varepsilon_2(k)$. Let $p_{\max}(k)$ be the largest prime number such that
\[p_{\max}(k)\leq k^{1/\varepsilon_1(k)}.\]  There are only finitely many primes $p$ such that $p\leq p_{\max}(k)$. For each such $p$, let
\[\mathscr{B}_p(k)=\left\{\left.\log_p\frac{k-1+m}{m}\,\right|\, m\in \mb{Z}^+, \varepsilon_1(k)\leq \log_p\frac{k-1+m}{m}\leq\varepsilon_2(k)\right\}.\]
Then 
\[\mathscr{B}(k)=\bigcup_{p\leq p_{\max}(k)}\mathscr{B}_p(k)\] is a finite set which contains all elements of  $\mathscr{A}(k)$ that lies in the interval $[\varepsilon_1(k), \varepsilon_2(k)]$. Arrange  the elements in $\mathscr{B}(k)$ in descending order. For each $\varepsilon\in \mathscr{B}(k)$, compute all the superior $k$-highly composite numbers $N_{k,\varepsilon}$ associated to $\varepsilon$, 
and compute the corresponding values $f(N_{k,\varepsilon})$. If any of the values $f(N_{k,\varepsilon})$ is larger than or equal to $\lambda_1(k)$, then the largest one gives $\lambda(k)$, and the corresponding $N_{k,\varepsilon}$ is then $N_{\max}(k)$.

In case none of the $f(N_{k,\varepsilon})$ with $\varepsilon\in\mathscr{B}(k)$ is larger than or equal to $\lambda_1(k)$, one changes the $\varepsilon_1(k)$ to a smaller one, and repeat the whole algorithm. Since $\lambda_1(k)\to 1$ as $\varepsilon_1(k)\to 0^+$, but we know that $\lambda(k)\geq 1.4677 $, this process must end after a finite number of trials.

 We implement the algorithm using MATLAB, and the results are given in the following.

\medskip
\subsection{The $\pmb{k=2}$ case}

Although this case has been solved in \cite{Nicolas}, we   reproduce it here for completeness.
We use 
\[\varepsilon_1(2)=\varepsilon_0(2)=0.1950\hspace{1cm}\text{and}\hspace{1cm}\varepsilon_2(2)=0.3466.\] Then \[\lambda_1(2)=1.5126.\] 
In Table \ref{table4}, we list down all the superior $2$-highly composite numbers $N_{2,\varepsilon}$ and the corresponding values $f_2(N_{2,\varepsilon})$ for  $\varepsilon_1(2)\leq\varepsilon\leq \varepsilon_2(2)$.

\begin{table}[h]
\caption{\label{table4} The values of $f_2(N_{2,\varepsilon})$ for $\varepsilon_1(2)\leq\varepsilon\leq \varepsilon_2(2)$.}

\begin{tabular}{|c|p{10cm}|c|}

\hline	Range of $\varepsilon$	&	$N_{2,\varepsilon}$	&	$f_2(N_{2,\varepsilon})$	\\\hline
$	0.3219\leq \varepsilon \leq 0.3466	$&$	2520=2^3\times 3^2\times 5\times 7	$&$	1.4677	$\\\hline
$	0.2891\leq\varepsilon\leq0.3219	$&$	5040=2^4\times 3^2\times 5\times 7	$&$	1.4849	$\\\hline
$	0.2702\leq\varepsilon\leq0.2891	$&$	55440=2^4\times 3^2\times 5\times 7\times 11	$&$	1.5118	$\\\hline
$	0.2630\leq\varepsilon\leq0.2702	$&$	720720=2^4\times 3^2\times 5\times 7\times 11\times 13	$&$	1.5252	$\\\hline
$	0.2619\leq\varepsilon\leq0.2630	$&$	1441440=2^5\times 3^2\times 5\times 7\times 11\times 13	$&$	1.5278	$\\\hline
$	0.2519\leq\varepsilon\leq0.2619	$&$	4324320=2^5\times 3^3\times 5\times 7\times 11\times 13	$&$	1.5319	$\\\hline
$	0.2447\leq\varepsilon\leq0.2519	$&$	21621600=2^5\times 3^3\times 5^2\times 7\times 11\times 13	$&$	1.5347	$\\\hline
$	0.2354\leq\varepsilon\leq0.2447	$&$	367567200=2^5\times 3^3\times 5^2\times 7\times 11\times 13\times 17	$&$	1.5376	$\\\hline
$	0.2224\leq\varepsilon\leq0.2354	$&$	6983776800=2^5\times 3^3\times 5^2\times 7\times 11\times 13\times 17\times 19	$&$	1.5379	$\\\hline
$	0.2211\leq\varepsilon\leq0.2224	$&$	13967553600=2^6\times 3^3\times 5^2\times 7\times 11\times 13\times 17\times 19	$&$	1.5367	$\\\hline
$	0.2084\leq\varepsilon\leq0.2211	$&$	321253732800=2^6\times 3^3\times 5^2\times 7\times 11\times 13\times 17\times 19\times 23	$&$	1.5327	$\\\hline
$	0.2058\leq\varepsilon\leq0.2084	$&$	2248776129600=2^6\times 3^3\times 5^2\times 7^2\times 11\times 13\times 17\times 19\times 23	$&$	1.5276	$\\\hline
 	\multirow{2}{8.6em}{$0.2031\leq\varepsilon\leq0.2058	$}&$	65214507758400$  & 	\multirow{2}{2.8em}{$1.5203	$} \\
	 & $=2^6\times 3^3\times 5^2\times 7^2\times 11\times 13\times 17\times 19\times 23\times 29	$& \\\hline
	\multirow{2}{8.6em}{$	0.2018\leq\varepsilon\leq0.2031	$}&$	195643523275200$  & 	\multirow{2}{2.8em}{$	1.5181	$}
	\\
	&$=2^6\times 3^4\times 5^2\times 7^2\times 11\times 13\times 17\times 19\times 23\times 29	$&\\\hline
	\multirow{2}{8.6em}{$	0.1950\leq \varepsilon\leq 0.2018	$}&$	6064949221531200$  &\multirow{2}{2.8em}{$	1.5126	$}\\
	&$=2^6\times 3^4\times 5^2\times 7^2\times 11\times 13\times 17\times 19\times 23\times 29\times 31	$&\\\hline

\end{tabular}
\end{table}

From the table, we find that the maximum value of $f_2(n)$ is 
\[\lambda(2)=1.5379,\] and it appears at
\[N_{\max}(2)=6983776800=2^5\times 3^3\times 5^2\times 7\times 11\times 13\times 17\times 19.\]
It is the superior 2-highly composite number associated to $\varepsilon$ for 
\[0.2224\leq\varepsilon\leq0.2354.\]

\medskip
\subsection{The $\pmb{k=3}$ case}

We use \[\varepsilon_1(3)=0.98\times\varepsilon_0(3)=0.3267\hspace{1cm} \text{and}\hspace{1cm}\varepsilon_2(3)=0.5493.\] Then 
\[\lambda_1(3)=1.5810.\] 
In Table \ref{table13}, we list down all the superior $3$-highly composite numbers $N_{3,\varepsilon}$ and the corresponding values $f_3(N_{3,\varepsilon})$ for  $\varepsilon_1(3)\leq\varepsilon\leq \varepsilon_2(3)$.

\begin{table}[h]
\caption{\label{table13} The values of $f_3(N_{3,\varepsilon})$ for $\varepsilon_1(3)\leq\varepsilon\leq \varepsilon_2(3)$.}

\begin{tabular}{|c|p{10cm}|c|}

\hline	Range of $\varepsilon$	&	$N_{3,\varepsilon}$	&	$f_3(N_{3,\varepsilon})$	\\\hline
$	0.4854\leq\varepsilon\leq0.5493	$&$	5040=2^4\times 3^2\times 5\times 7	$&$	1.5324	$\\\hline
$	0.4650\leq\varepsilon\leq0.4854	$&$	10080=2^5\times 3^2\times 5\times 7	$&$	1.5426	$\\\hline
$	0.4582\leq\varepsilon\leq0.4650	$&$	30240=2^5\times 3^3\times 5\times 7	$&$	1.5534	$\\\hline
$	0.4307\leq\varepsilon\leq0.4582	$&$	332640=2^5\times 3^3\times 5\times 7\times 11	$&$	1.5733	$\\\hline
$	0.4283\leq\varepsilon\leq0.4307	$&$	1663200=2^5\times 3^3\times 5^2\times 7\times 11	$&$	1.5792	$\\\hline
$	0.4150\leq\varepsilon\leq0.4283	$&$	21621600=2^5\times 3^3\times 5^2\times 7\times 11\times 13	$&$	1.5897	$\\\hline
$	0.3878\leq\varepsilon\leq0.4150	$&$	43243200=2^6\times 3^3\times 5^2\times 7\times 11\times 13	$&$	1.5914	$\\\hline
$	0.3731\leq\varepsilon\leq0.3878	$&$	735134400=2^6\times 3^3\times 5^2\times 7\times 11\times 13\times 17	$&$	1.5897	$\\\hline
$	0.3691\leq\varepsilon\leq0.3731	$&$	13967553600=2^6\times 3^3\times 5^2\times 7\times 11\times 13\times 17\times 19	$&$	1.5863	$\\\hline
$	0.3626\leq\varepsilon\leq0.3691	$&$	41902660800=2^6\times 3^4\times 5^2\times 7\times 11\times 13\times 17\times 19	$&$	1.5854	$\\\hline
$	0.3562\leq\varepsilon\leq0.3626	$&$	83805321600=2^7\times 3^4\times 5^2\times 7\times 11\times 13\times 17\times 19	$&$	1.5845	$\\\hline
$	0.3504\leq\varepsilon\leq0.3562	$&$	586637251200=2^7\times 3^4\times 5^2\times 7^2\times 11\times 13\times 17\times 19	$&$	1.5815	$\\\hline
$	0.3267\leq\varepsilon\leq0.3504	$&$	13492656777600=2^7\times 3^4\times 5^2\times 7^2\times 11\times 13\times 17\times 19\times 23	$&$	1.5773	$\\\hline

\end{tabular}
\end{table}

From the table, we find that the maximum value of $f_3(n)$ is 
\[\lambda(3)=1.5914,\] and it appears at
\[N_{\max}(3)=43243200=2^6\times 3^3\times 5^2\times 7\times 11\times 13.\]
It is the superior 3-highly composite number associated to $\varepsilon$ for 
\[0.3878\leq\varepsilon\leq0.4150.\]

\vfill\pagebreak
\subsection{The $\pmb{k=4}$ case}
We use \[\varepsilon_1(4)=0.94\times\varepsilon_0(4)=0.4142\hspace{1cm} \text{and}\hspace{1cm}\varepsilon_2(4)=0.6931.\] Then 
\[\lambda_1(4)=1.6265.\] 
In Table \ref{table14}, we list down all the superior $4$-highly composite numbers $N_{4,\varepsilon}$ and the corresponding values $f_4(N_{4,\varepsilon})$ for  $\varepsilon_1(4)\leq\varepsilon\leq \varepsilon_2(4)$.

\begin{table}[h]
\caption{\label{table14} The values of $f_4(N_{4,\varepsilon})$ for $\varepsilon_1(4)\leq\varepsilon\leq \varepsilon_2(4)$.}

\begin{tabular}{|c|p{10cm}|c|}

\hline	Range of $\varepsilon$	&	$N_{4,\varepsilon}$	&	$f_4(N_{4,\varepsilon})$	\\\hline
$	0.6781\leq\varepsilon\leq0.6931	$&$	5040=2^4\times 3^2\times 5\times 7	$&$	1.5650	$\\\hline
$	0.6309\leq\varepsilon\leq0.6781	$&$	10080=2^5\times 3^2\times 5\times 7	$&$	1.5818	$\\\hline
$	0.5850\leq\varepsilon\leq0.6309	$&$	30240=2^5\times 3^3\times 5\times 7	$&$	1.5981	$\\\hline
$	0.5781\leq\varepsilon\leq0.5850	$&$	60480=2^6\times 3^3\times 5\times 7	$&$	1.6029	$\\\hline
$	0.5693\leq\varepsilon\leq0.5781	$&$	665280=2^6\times 3^3\times 5\times 7\times 11	$&$	1.6180	$\\\hline
$	0.5405\leq\varepsilon\leq0.5693	$&$	3326400=2^6\times 3^3\times 5^2\times 7\times 11	$&$	1.6269	$\\\hline
$	0.5146\leq\varepsilon\leq0.5405	$&$	43243200=2^6\times 3^3\times 5^2\times 7\times 11\times 13	$&$	1.6335	$\\\hline
$	0.5094\leq\varepsilon\leq0.5146	$&$	86486400=2^7\times 3^3\times 5^2\times 7\times 11\times 13	$&$	1.6336	$\\\hline
$	0.4893\leq\varepsilon\leq0.5094	$&$	259459200=2^7\times 3^4\times 5^2\times 7\times 11\times 13	$&$	1.6337	$\\\hline
$	0.4709\leq\varepsilon\leq0.4893	$&$	4410806400=2^7\times 3^4\times 5^2\times 7\times 11\times 13\times 17	$&$	1.6305	$\\\hline
$	0.4708\leq\varepsilon\leq0.4709	$&$	30875644800=2^7\times 3^4\times 5^2\times 7^2\times 11\times 13\times 17	$&$	1.6269	$\\\hline
$	0.4594\leq\varepsilon\leq0.4708	$&$	586637251200=2^7\times 3^4\times 5^2\times 7^2\times 11\times 13\times 17\times 19	$&$	1.6243	$\\\hline
$	0.4421\leq\varepsilon\leq0.4594	$&$	1173274502400=2^8\times 3^4\times 5^2\times 7^2\times 11\times 13\times 17\times 19	$&$	1.6234	$\\\hline
$	0.4142\leq\varepsilon\leq0.4421	$&$	26985313555200=2^8\times 3^4\times 5^2\times 7^2\times 11\times 13\times 17\times 19\times 23	$&$	1.6167	$\\\hline

\end{tabular}
\end{table}

From the table, we find that the maximum value of $f_4(n)$ is 
\[\lambda(4)=	1.6337,\] and it appears at
\[N_{\max}(4)=259459200=2^7\times 3^4\times 5^2\times 7\times 11\times 13.\]
It is the superior 4-highly composite number associated to $\varepsilon$ for 
\[0.4893\leq\varepsilon\leq0.5094.\]

\vfill\pagebreak
\subsection{The $\pmb{k=5}$ case}
We use \[\varepsilon_1(5)=0.92\times\varepsilon_0(5)=0.4861\hspace{1cm} \text{and}\hspace{1cm}\varepsilon_2(5)=0.8047.\] Then 
\[\lambda_1(5)=1.6657.\] 
In Table \ref{table15}, we list down all the superior $5$-highly composite numbers $N_{5,\varepsilon}$ and the corresponding values $f_5(N_{5,\varepsilon})$ for  $\varepsilon_1(5)\leq\varepsilon\leq \varepsilon_2(5)$.

\begin{table}[h]
\caption{\label{table15} The values of $f_5(N_{5,\varepsilon})$ for $\varepsilon_1(5)\leq\varepsilon\leq \varepsilon_2(5)$.}

\begin{tabular}{|c|p{10cm}|c|}

\hline	Range of $\varepsilon$	&	$N_{5,\varepsilon}$	&	$f_5(N_{5,\varepsilon})$	\\\hline
  $	0.7712\leq\varepsilon\leq0.8047	$&$	10080=2^4\times 3^2\times 5\times 7	$&$	1.6114	$\\\hline
$	0.7370\leq\varepsilon\leq0.7712	$&$	30240=2^5\times 3^3\times 5\times 7	$&$	1.6319	$\\\hline
$	0.6826\leq\varepsilon\leq0.7370	$&$	60480=2^6\times 3^3\times 5\times 7	$&$	1.6409	$\\\hline
$	0.6712\leq\varepsilon\leq0.6826	$&$	302400=2^6\times 3^3\times 5^2\times 7	$&$	1.6502	$\\\hline
$	0.6521\leq\varepsilon\leq0.6712	$&$	3326400=2^6\times 3^3\times 5^2\times 7\times 11	$&$	1.6623	$\\\hline
$	0.6309\leq\varepsilon\leq0.6521	$&$	6652800=2^7\times 3^3\times 5^2\times 7\times 11	$&$	1.6646	$\\\hline
$	0.6275\leq\varepsilon\leq0.6309	$&$	19958400=2^7\times 3^4\times 5^2\times 7\times 11	$&$	1.6663	$\\\hline
$	0.5850\leq\varepsilon\leq0.6275	$&$	259459200=2^7\times 3^4\times 5^2\times 7\times 11\times 13	$&$	1.6714	$\\\hline
$	0.5681\leq\varepsilon\leq0.5850  $&$	518918400=2^8\times 3^4\times 5^2\times 7\times 11\times 13	$&$	1.6705	$\\\hline
$	0.5646\leq\varepsilon\leq0.5681	$&$	8821612800=2^8\times 3^4\times 5^2\times 7\times 11\times 13\times 17	$&$	1.6650	$\\\hline
$	0.5466\leq\varepsilon\leq0.5646	$&$	61751289600=2^8\times 3^4\times 5^2\times 7^2\times 11\times 13\times 17	$&$	1.6628	$\\\hline
$	0.5350\leq\varepsilon\leq0.5466	$&$	1173274502400=2^8\times 3^4\times 5^2\times 7^2\times 11\times 13\times 17\times 19	$&$	1.6581	$\\\hline
$	0.5305\leq\varepsilon\leq0.5350	$&$	3519823507200=2^8\times 3^5\times 5^2\times 7^2\times 11\times 13\times 17\times 19	$&$	1.6562	$\\\hline
$	0.5265\leq\varepsilon\leq0.5305	$&$	7039647014400=2^9\times 3^5\times 5^2\times 7^2\times 11\times 13\times 17\times 19	$&$	1.6549	$\\\hline
$	0.5133\leq\varepsilon\leq0.5265	$&$	35198235072000=2^9\times 3^5\times 5^3\times 7^2\times 11\times 13\times 17\times 19	$&$	1.6522	$\\\hline
$	0.4861\leq\varepsilon\leq0.5133	$&$	809559406656000=2^9\times 3^5\times 5^3\times 7^2\times 11\times 13\times 17\times 19\times 23	$&$	1.6461	$\\\hline

\end{tabular}
\end{table}

From the table, we find that the maximum value of $f_5(n)$ is 
\[\lambda(5)=1.6714,\] and it appears at
\[N_{\max}(5)=259459200=2^7\times 3^4\times 5^2\times 7\times 11\times 13.\]
It is the superior 5-highly composite number associated to $\varepsilon$ for 
\[0.5850\leq\varepsilon\leq0.6275.\]
Note that $N_{\max}(5)=N_{\max}(4)$ even though $\lambda(5)\neq\lambda(4)$.

\vfill\pagebreak
\subsection{The general case}
We compute $\lambda(k)$, the maximum value of $f_k(n)$, for $2\leq k\leq 2000$ and the corresponding maximizer $N_{\max}(k)$. The results for $\lambda(k)$ as well as the corresponding parameters  when $2\leq k\leq 100$ were given in Tables \ref{table16_1},   \ref{table16_2} and \ref{table16_3}. The range of $\varepsilon$ is such that $N_{\max}(k)=N_{k,\varepsilon}$. 

\renewcommand{\hsp}{\hspace{0.2cm}}
 \begin{table}[h]
\caption{\label{table16_1} The values of $\lambda(k)$ and the corresponding parameters (Part I).}

\begin{tabular}{|c|c|c|c|c|c|}
\hline
$\hsp k\hsp $& $\varepsilon_1(k)$ & $\varepsilon_2(k)$ & $\lambda_1(k)$ & Range of $\varepsilon$ &  $\lambda(k)$
\\\hline
2&$\hsp	0.1950=  \varepsilon_0(k)	\hsp $&$\hsp	0.3466\hsp	$&$\hsp	1.5126\hsp	$&$	\hsp 0.2224\leq \varepsilon\leq 0.2354\hsp 	$&$	\hsp 1.5379	\hsp$\\\hline
3&$	0.3267=0.98\times \varepsilon_0(k)	$&$	0.5493	$&$	1.5882	$&$	0.3878\leq \varepsilon\leq 0.4150	$&$	1.5914	$\\\hline
4&$	0.4142=0.94\times \varepsilon_0(k)	$&$	0.6931	$&$	1.6265	$&$	0.4893\leq \varepsilon\leq 0.5094	$&$	1.6337	$\\\hline
5&$	0.4861=0.92\times \varepsilon_0(k)	$&$	0.8047	$&$	1.6657	$&$	0.5850\leq \varepsilon\leq 0.6275	$&$	1.6714	$\\\hline
6&$	0.5422=0.9\times \varepsilon_0(k)	$&$	0.8959	$&$	1.6956	$&$	0.6438\leq \varepsilon\leq 0.6986	$&$	1.7029	$\\\hline
7&$	0.5933=0.89\times \varepsilon_0(k)	$&$	0.9730	$&$	1.7276	$&$	0.7370\leq \varepsilon\leq 0.7587	$&$	1.7299	$\\\hline
8&$	0.6365=0.88\times \varepsilon_0(k)	$&$	1.0397	$&$	1.7547	$&$	0.8107\leq \varepsilon\leq 0.8301	$&$	1.7549	$\\\hline
9&$	0.6734=0.87\times \varepsilon_0(k)	$&$	1.0986	$&$	1.7778	$&$	0.8697\leq \varepsilon\leq 0.9163	$&$	1.7782	$\\\hline
10&$	0.7050=0.86\times \varepsilon_0(k)	$&$	1.1513	$&$	1.7977	$&$	0.8977\leq \varepsilon\leq 0.9260	$&$	1.7993	$\\\hline
11&$	0.7324=0.85\times \varepsilon_0(k)	$&$	1.1989	$&$	1.8147	$&$	0.9349\leq \varepsilon\leq 1	$&$	1.8198	$\\\hline
12&$	0.7561=0.84\times \varepsilon_0(k)	$&$	1.2425	$&$	1.8294	$&$	1\leq \varepsilon\leq 1.0363	$&$	1.8382	$\\\hline
13&$	0.7861=0.84\times \varepsilon_0(k)	$&$	1.2825	$&$	1.8529	$&$	1\leq \varepsilon\leq 1.0641	$&$	1.8553	$\\\hline
14&$	0.8042=0.83\times \varepsilon_0(k)	$&$	1.3195	$&$	1.8638	$&$	1.1006\leq \varepsilon\leq 1.1255	$&$	1.8719	$\\\hline
15&$	0.8300=0.83\times \varepsilon_0(k)	$&$	1.3540	$&$	1.8844	$&$	1.1293\leq \varepsilon\leq 1.1844	$&$	1.8880	$\\\hline
16&$	0.8439=0.82\times \varepsilon_0(k)	$&$	1.3863	$&$	1.8923	$&$	1.1699\leq \varepsilon\leq 1.2410	$&$	1.9030	$\\\hline
17&$	0.8664=0.82\times \varepsilon_0(k)	$&$	1.4166	$&$	1.9107	$&$	1.1827\leq \varepsilon\leq 1.2224	$&$	1.9169	$\\\hline
18&$	0.8878=0.82\times \varepsilon_0(k)	$&$	1.4452	$&$	1.9282	$&$	1.2231\leq \varepsilon\leq 1.2730	$&$	1.9305	$\\\hline
19&$	0.8969=0.81\times \varepsilon_0(k)	$&$	1.4722	$&$	1.9326	$&$	1.2619\leq \varepsilon\leq 1.3219	$&$	1.9433	$\\\hline
20&$	0.9159=0.81\times \varepsilon_0(k)	$&$	1.4979	$&$	1.9484	$&$	1.2996\leq \varepsilon\leq 1.3692	$&$	1.9553	$\\\hline
21&$	0.9340=0.81\times \varepsilon_0(k)	$&$	1.5223	$&$	1.9636	$&$	1.2801\leq \varepsilon\leq 1.3347	$&$	1.9670	$\\\hline
22&$	0.9514=0.81\times \varepsilon_0(k)	$&$	1.5455	$&$	1.9782	$&$	1.3219\leq \varepsilon\leq 1.3691	$&$	1.9786	$\\\hline
23&$	0.9560=0.8\times \varepsilon_0(k)	$&$	1.5677	$&$	1.9793	$&$	1.3626\leq \varepsilon\leq 1.4022	$&$	1.9895	$\\\hline
24&$	0.9717=0.8\times \varepsilon_0(k)	$&$	1.5890	$&$	1.9926	$&$	1.3418\leq \varepsilon\leq 1.4021	$&$	2.0000	$\\\hline
25&$	0.9868=0.8\times \varepsilon_0(k)	$&$	1.6094	$&$	2.0055	$&$	1.3785\leq \varepsilon\leq 1.4406	$&$	2.0103	$\\\hline
26&$	1.0013=0.8\times \varepsilon_0(k)	$&$	1.7105	$&$	2.0179	$&$	1.4150\leq \varepsilon\leq 1.4780	$&$	2.0201	$\\\hline
27&$	1.0026=0.79\times \varepsilon_0(k)	$&$	1.7500	$&$	2.0164	$&$	1.4507\leq \varepsilon\leq 1.5146	$&$	2.0295	$\\\hline
28&$	1.0159=0.79\times \varepsilon_0(k)	$&$	1.7885	$&$	2.0279	$&$	1.4854\leq \varepsilon\leq 1.5502	$&$	2.0385	$\\\hline
29&$	1.0288=0.79\times \varepsilon_0(k)	$&$	1.8260	$&$	2.0391	$&$	1.4650\leq \varepsilon\leq 1.5194	$&$	2.0473	$\\\hline
30&$	1.0413=0.79\times \varepsilon_0(k)	$&$	1.8625	$&$	2.0499	$&$	1.4919\leq \varepsilon\leq 1.5525	$&$	2.0559	$\\\hline
31&$	1.0533=0.79\times \varepsilon_0(k)	$&$	1.8981	$&$	2.0604	$&$	1.5236\leq \varepsilon\leq 1.5850	$&$	2.0642	$\\\hline
32&$	1.0650=0.79\times \varepsilon_0(k)	$&$	1.9329	$&$	2.0707	$&$	1.5546\leq \varepsilon\leq 1.6167	$&$	2.0722	$\\\hline
33&$	1.0627=0.78\times \varepsilon_0(k)	$&$	1.9668	$&$	2.0663	$&$	1.5850\leq \varepsilon\leq 1.6477	$&$	2.0798	$\\\hline
34&$	1.0736=0.78\times \varepsilon_0(k)	$&$	2.0000	$&$	2.0759	$&$	1.5865\leq \varepsilon\leq 1.6147	$&$	2.0875	$\\\hline
\end{tabular}\end{table}

\vfill\pagebreak
~

 \begin{table}[h]
\caption{\label{table16_2} The values of $\lambda(k)$ and the corresponding parameters (Part II).}

\begin{tabular}{|c|c|c|c|c|c|}
\hline
$\hsp k\hsp $& $\varepsilon_1(k)$ & $\varepsilon_2(k)$ & $\lambda_1(k)$ & Range of $\varepsilon$ &  $\lambda(k)$
\\\hline

35&$\hsp	1.0842=0.78\times \varepsilon_0(k)\hsp	$&$	\hsp 2.0324\hsp	$&$\hsp	2.0852	\hsp$&$\hsp	1.6090\leq \varepsilon\leq 1.6439	\hsp$&$	\hsp 2.0949\hsp	$\\\hline
36&$	1.0945=0.78\times \varepsilon_0(k)	$&$	2.0641	$&$	2.0943	$&$	1.6309\leq \varepsilon\leq 1.6724	$&$	2.1021	$\\\hline
37&$	1.1045=0.78\times \varepsilon_0(k)	$&$	2.0952	$&$	2.1031	$&$	1.5937\leq \varepsilon\leq 1.6405	$&$	2.1092	$\\\hline
38&$	1.1142=0.78\times \varepsilon_0(k)	$&$	2.1255	$&$	2.1118	$&$	1.6114\leq \varepsilon\leq 1.6674	$&$	2.1162	$\\\hline
39&$	1.1237=0.78\times \varepsilon_0(k)	$&$	2.1553	$&$	2.1202	$&$	1.6374\leq \varepsilon\leq 1.6937	$&$	2.1231	$\\\hline
40&$	1.1330=0.78\times \varepsilon_0(k)	$&$	2.1844	$&$	2.1285	$&$	1.6630\leq \varepsilon\leq 1.7137	$&$	2.1298	$\\\hline
41&$	1.1274=0.77\times \varepsilon_0(k)	$&$	2.2130	$&$	2.1214	$&$	1.6881\leq \varepsilon\leq 1.7333	$&$	2.1362	$\\\hline
42&$	1.1361=0.77\times \varepsilon_0(k)	$&$	2.2410	$&$	2.1292	$&$	1.7127\leq \varepsilon\leq 1.7525	$&$	2.1425	$\\\hline
43&$	1.1446=0.77\times \varepsilon_0(k)	$&$	2.2685	$&$	2.1368	$&$	1.6828\leq \varepsilon\leq 1.7370	$&$	2.1488	$\\\hline
44&$	1.1529=0.77\times \varepsilon_0(k)	$&$	2.2955	$&$	2.1443	$&$	1.7063\leq \varepsilon\leq 1.7608	$&$	2.1549	$\\\hline
45&$	1.1610=0.77\times \varepsilon_0(k)	$&$	2.3219	$&$	2.1516	$&$	1.7294\leq \varepsilon\leq 1.7843	$&$	2.1608	$\\\hline
46&$	1.1690=0.77\times \varepsilon_0(k)	$&$	2.3479	$&$	2.1588	$&$	1.7521\leq \varepsilon\leq 1.8074	$&$	2.1667	$\\\hline
47&$	1.1768=0.77\times \varepsilon_0(k)	$&$	2.3735	$&$	2.1658	$&$	1.7744\leq \varepsilon\leq 1.8301	$&$	2.1723	$\\\hline
48&$	1.1844=0.77\times \varepsilon_0(k)	$&$	2.3985	$&$	2.1727	$&$	1.7965\leq \varepsilon\leq 1.8524	$&$	2.1778	$\\\hline
49&$	1.1919=0.77\times \varepsilon_0(k)	$&$	2.4232	$&$	2.1795	$&$	1.7712\leq \varepsilon\leq 1.8182	$&$	2.1834	$\\\hline
50&$	1.1993=0.77\times \varepsilon_0(k)	$&$	2.4475	$&$	2.1861	$&$	1.7874\leq \varepsilon\leq 1.8395	$&$	2.1888	$\\\hline
51&$	1.2064=0.77\times \varepsilon_0(k)	$&$	2.4713	$&$	2.1926	$&$	1.8074\leq \varepsilon\leq 1.8606	$&$	2.1941	$\\\hline
52&$	1.2135=0.77\times \varepsilon_0(k)	$&$	2.4948	$&$	2.1990	$&$	1.8278\leq \varepsilon\leq 1.8814	$&$	2.1992	$\\\hline
53&$	1.2046=0.76\times \varepsilon_0(k)	$&$	2.5178	$&$	2.1891	$&$	1.8480\leq \varepsilon\leq 1.9018	$&$	2.2043	$\\\hline
54&$	1.2113=0.76\times \varepsilon_0(k)	$&$	2.5406	$&$	2.1952	$&$	1.8679\leq \varepsilon\leq 1.9220	$&$	2.2092	$\\\hline
55&$	1.2179=0.76\times \varepsilon_0(k)	$&$	2.5629	$&$	2.2013	$&$	1.8639\leq \varepsilon\leq 1.8875	$&$	2.2141	$\\\hline
56&$	1.2243=0.76\times \varepsilon_0(k)	$&$	2.5850	$&$	2.2072	$&$	1.8785\leq \varepsilon\leq 1.9069	$&$	2.2190	$\\\hline
57&$	1.2307=0.76\times \varepsilon_0(k)	$&$	2.6067	$&$	2.2130	$&$	1.8928\leq \varepsilon\leq 1.9260	$&$	2.2238	$\\\hline
58&$	1.2369=0.76\times \varepsilon_0(k)	$&$	2.6280	$&$	2.2187	$&$	1.9069\leq \varepsilon\leq 1.9449	$&$	2.2284	$\\\hline
59&$	1.2431=0.76\times \varepsilon_0(k)	$&$	2.6491	$&$	2.2243	$&$	1.8716\leq \varepsilon\leq 1.9115	$&$	2.2330	$\\\hline
60&$	1.2491=0.76\times \varepsilon_0(k)	$&$	2.6699	$&$	2.2298	$&$	1.8817\leq \varepsilon\leq 1.9296	$&$	2.2377	$\\\hline
61&$	1.2551=0.76\times \varepsilon_0(k)	$&$	2.6903	$&$	2.2353	$&$	1.8981\leq \varepsilon\leq 1.9475	$&$	2.2422	$\\\hline
62&$	1.2609=0.76\times \varepsilon_0(k)	$&$	2.7105	$&$	2.2406	$&$	1.9156\leq \varepsilon\leq 1.9613	$&$	2.2467	$\\\hline
63&$	1.2667=0.76\times \varepsilon_0(k)	$&$	2.7304	$&$	2.2459	$&$	1.9329\leq \varepsilon\leq 1.9744	$&$	2.2511	$\\\hline
64&$	1.2723=0.76\times \varepsilon_0(k)	$&$	2.7500	$&$	2.2511	$&$	1.9500\leq \varepsilon\leq 1.9873	$&$	2.2554	$\\\hline
65&$	1.2779=0.76\times \varepsilon_0(k)	$&$	2.7694	$&$	2.2563	$&$	1.9668\leq \varepsilon\leq 2	$&$	2.2596	$\\\hline
66&$	1.2834=0.76\times \varepsilon_0(k)	$&$	2.7885	$&$	2.2613	$&$	1.9391\leq \varepsilon\leq 1.9835	$&$	2.2638	$\\\hline
67&$	1.2888=0.76\times \varepsilon_0(k)	$&$	2.8074	$&$	2.2663	$&$	1.9522\leq \varepsilon\leq 2	$&$	2.2680	$\\\hline

\end{tabular}\end{table}

\vfill\pagebreak
~

 \begin{table}[h]
\caption{\label{table16_3}The values of $\lambda(k)$ and the corresponding parameters (Part III).}

\begin{tabular}{|c|c|c|c|c|c|}
\hline
$\hsp k\hsp$& $\varepsilon_1(k)$ & $\varepsilon_2(k)$ & $\lambda_1(k)$ & Range of $\varepsilon$ &  $\lambda(k)$
\\\hline
68&$\hsp	1.2942=0.76\times \varepsilon_0(k)	\hsp$&$	\hsp 2.8260\hsp	$&$	\hsp 2.2712\hsp	$&$\hsp	1.9683\leq \varepsilon\leq 2.0163	\hsp $&$	\hsp 2.2721\hsp	$\\\hline
69&$	1.2994=0.76\times \varepsilon_0(k)	$&$	2.8443	$&$	2.2761	$&$	1.9842\leq \varepsilon\leq 2.0324	$&$	2.2762	$\\\hline
70&$	1.2874=0.75\times \varepsilon_0(k)	$&$	2.8625	$&$	2.2635	$&$	2\leq \varepsilon\leq 2.0484	$&$	2.2801	$\\\hline
71&$	1.2925=0.75\times \varepsilon_0(k)	$&$	2.8804	$&$	2.2682	$&$	2.0156\leq \varepsilon\leq 2.0641	$&$	2.2840	$\\\hline
72&$	1.2975=0.75\times \varepsilon_0(k)	$&$	2.8981	$&$	2.2728	$&$	2.0310\leq \varepsilon\leq 2.0797	$&$	2.2879	$\\\hline
73&$	1.3024=0.75\times \varepsilon_0(k)	$&$	2.9156	$&$	2.2774	$&$	2\leq \varepsilon\leq 2.0463	$&$	2.2916	$\\\hline
74&$	1.3072=0.75\times \varepsilon_0(k)	$&$	2.9329	$&$	2.2818	$&$	2.0150\leq \varepsilon\leq 2.0614	$&$	2.2954	$\\\hline
75&$	1.3120=0.75\times \varepsilon_0(k)	$&$	2.9500	$&$	2.2863	$&$	2.0297\leq \varepsilon\leq 2.0764	$&$	2.2992	$\\\hline
76&$	1.3167=0.75\times \varepsilon_0(k)	$&$	2.9668	$&$	2.2906	$&$	2.0444\leq \varepsilon\leq 2.0911	$&$	2.3029	$\\\hline
77&$	1.3214=0.75\times \varepsilon_0(k)	$&$	2.9835	$&$	2.2950	$&$	2.0589\leq \varepsilon\leq 2.1058	$&$	2.3065	$\\\hline
78&$	1.3259=0.75\times \varepsilon_0(k)	$&$	3.0000	$&$	2.2992	$&$	2.0732\leq \varepsilon\leq 2.1203	$&$	2.3101	$\\\hline
79&$	1.3305=0.75\times \varepsilon_0(k)	$&$	3.0163	$&$	2.3034	$&$	2.0875\leq \varepsilon\leq 2.1346	$&$	2.3136	$\\\hline
80&$	1.3350=0.75\times \varepsilon_0(k)	$&$	3.0324	$&$	2.3076	$&$	2.1015\leq \varepsilon\leq 2.1489	$&$	2.3171	$\\\hline
81&$	1.3394=0.75\times \varepsilon_0(k)	$&$	3.0484	$&$	2.3117	$&$	2.0857\leq \varepsilon\leq 2.1155	$&$	2.3205	$\\\hline
82&$	1.3438=0.75\times \varepsilon_0(k)	$&$	3.0641	$&$	2.3158	$&$	2.0959\leq \varepsilon\leq 2.1293	$&$	2.3239	$\\\hline
83&$	1.3481=0.75\times \varepsilon_0(k)	$&$	3.0797	$&$	2.3198	$&$	2.1060\leq \varepsilon\leq 2.1430	$&$	2.3273	$\\\hline
84&$	1.3523=0.75\times \varepsilon_0(k)	$&$	3.0952	$&$	2.3238	$&$	2.1159\leq \varepsilon\leq 2.1565	$&$	2.3307	$\\\hline
85&$	1.3566=0.75\times \varepsilon_0(k)	$&$	3.1104	$&$	2.3277	$&$	2.1257\leq \varepsilon\leq 2.1699	$&$	2.3340	$\\\hline
86&$	1.3607=0.75\times \varepsilon_0(k)	$&$	3.1255	$&$	2.3316	$&$	2.1375\leq \varepsilon\leq 2.1832	$&$	2.3372	$\\\hline
87&$	1.3649=0.75\times \varepsilon_0(k)	$&$	3.1405	$&$	2.3355	$&$	2.1506\leq \varepsilon\leq 2.1964	$&$	2.3404	$\\\hline
88&$	1.3689=0.75\times \varepsilon_0(k)	$&$	3.1553	$&$	2.3393	$&$	2.1635\leq \varepsilon\leq 2.2095	$&$	2.3436	$\\\hline
89&$	1.3730=0.75\times \varepsilon_0(k)	$&$	3.1699	$&$	2.3430	$&$	2.1641\leq \varepsilon\leq 2.1763	$&$	2.3467	$\\\hline
90&$	1.3769=0.75\times \varepsilon_0(k)	$&$	3.1844	$&$	2.3468	$&$	2.1451\leq \varepsilon\leq 2.1734	$&$	2.3499	$\\\hline
91&$	1.3809=0.75\times \varepsilon_0(k)	$&$	3.1988	$&$	2.3504	$&$	2.1575\leq \varepsilon\leq 2.1827	$&$	2.3530	$\\\hline
92&$	1.3848=0.75\times \varepsilon_0(k)	$&$	3.2130	$&$	2.3541	$&$	2.1699\leq \varepsilon\leq 2.1918	$&$	2.3561	$\\\hline
93&$	1.3886=0.75\times \varepsilon_0(k)	$&$	3.2271	$&$	2.3577	$&$	2.1822\leq \varepsilon\leq 2.2009	$&$	2.3592	$\\\hline
94&$	1.3925=0.75\times \varepsilon_0(k)	$&$	3.2410	$&$	2.3613	$&$	2.1944\leq \varepsilon\leq 2.2098	$&$	2.3622	$\\\hline
95&$	1.3962=0.75\times \varepsilon_0(k)	$&$	3.2548	$&$	2.3648	$&$	2.1640\leq \varepsilon\leq 2.2065	$&$	2.3652	$\\\hline
96&$	1.3813=0.74\times \varepsilon_0(k)	$&$	3.2685	$&$	2.3496	$&$	2.1758\leq \varepsilon\leq 2.2184	$&$	2.3682	$\\\hline
97&$	1.3850=0.74\times \varepsilon_0(k)	$&$	3.2820	$&$	2.3531	$&$	2.1876\leq \varepsilon\leq 2.2303	$&$	2.3712	$\\\hline
98&$	1.3886=0.74\times \varepsilon_0(k)	$&$	3.2955	$&$	2.3565	$&$	2.1993\leq \varepsilon\leq 2.2421	$&$	2.3741	$\\\hline
99&$	1.3922=0.74\times \varepsilon_0(k)	$&$	3.3088	$&$	2.3598	$&$	2.2109\leq \varepsilon\leq 2.2534	$&$	2.3770	$\\\hline
100&$	1.3957=0.74\times \varepsilon_0(k)	$&$	3.3219	$&$	2.3631	$&$	2.2224\leq \varepsilon\leq 2.2619	$&$	2.3799	$\\\hline

\end{tabular}
\end{table}

\vfill\pagebreak

The values of $N_{\max}(k)$ for $2\leq k\leq 2000$ are given in Tables \ref{table17_1} and \ref{table17_2}. 

\renewcommand{\hsp}{\hspace{0.5cm}}
 \begin{table}[h]
\caption{\label{table17_1} The values of $N_{\max}(k)$ (Part I).}

\begin{tabular}{|c|l|}
\hline
Range of $k$ & $N_{\max}(k)$\\\hline
$	k=2	$&$	6983776800=2^{5}\times 3^{3}\times 5^2\times 7\times 11\times 13\times 17\times 19\hspace{1cm}	$\\\hline
$	k=3	$&$	43243200=2^{6}\times 3^{3}\times 5^2\times 7\times 11\times 13	$\\\hline
$	4\leq k\leq 5	$&$	259459200=2^{7}\times 3^{4}\times 5^2\times 7\times 11\times 13	$\\\hline
$	6\leq k\leq 7	$&$	518918400=2^{8}\times 3^{4}\times 5^2\times 7\times 11\times 13	$\\\hline
$	8\leq k\leq 9	$&$	79833600=2^{9}\times 3^{4}\times 5^2\times 7\times 11	$\\\hline
$	10\leq k\leq 12	$&$	479001600=2^{10}\times 3^{5}\times 5^2\times 7\times 11	$\\\hline
$	k=13	$&$	958003200=2^{11}\times 3^{5}\times 5^2\times 7\times 11	$\\\hline
$	14\leq k\leq 16	$&$	87091200=2^{11}\times 3^{5}\times 5^2\times 7	$\\\hline
$	17\leq k\leq 20	$&$	174182400=2^{12}\times 3^{5}\times 5^2\times 7	$\\\hline
$	21\leq k\leq 23	$&$	1045094400=2^{13}\times 3^{6}\times 5^2\times 7	$\\\hline
$	24\leq k\leq 28	$&$	2090188800=2^{14}\times 3^{6}\times 5^2\times 7	$\\\hline
$	29\leq k\leq 33	$&$	4180377600=2^{15}\times 3^{6}\times 5^2\times 7	$\\\hline
$	34\leq k\leq 36	$&$	8360755200=2^{16}\times 3^{6}\times 5^2\times 7	$\\\hline
$	37\leq k\leq 42	$&$	50164531200=2^{17}\times 3^{7}\times 5^2\times 7	$\\\hline
$	43\leq k\leq 48	$&$	100329062400=2^{18}\times 3^{7}\times 5^2\times 7	$\\\hline
$	49\leq k\leq 54	$&$	200658124800=2^{19}\times 3^{7}\times 5^2\times 7	$\\\hline
$	55\leq k\leq 58	$&$	401316249600=2^{20}\times 3^{7}\times 5^2\times 7	$\\\hline
$	59\leq k\leq 65	$&$	2407897497600=2^{21}\times 3^{8}\times 5^2\times 7	$\\\hline
$	66\leq k\leq 72	$&$	4815794995200=2^{22}\times 3^{8}\times 5^2\times 7	$\\\hline
$	73\leq k\leq 80	$&$	9631589990400=2^{23}\times 3^{8}\times 5^2\times 7	$\\\hline
$	81\leq k\leq 88	$&$	19263179980800=2^{24}\times 3^{8}\times 5^2\times 7	$\\\hline
$	89\leq k\leq 89	$&$	38526359961600=2^{25}\times 3^{8}\times 5^2\times 7	$\\\hline
$	90\leq k\leq 94	$&$	115579079884800=2^{25}\times 3^{9}\times 5^2\times 7	$\\\hline
$	95\leq k\leq 102	$&$	231158159769600=2^{26}\times 3^{9}\times 5^2\times 7	$\\\hline
$	103\leq k\leq 111	$&$	462316319539200=2^{27}\times 3^{9}\times 5^2\times 7	$\\\hline
$	112\leq k\leq 121	$&$	924632639078400=2^{28}\times 3^{9}\times 5^2\times 7	$\\\hline
$	122\leq k\leq 131	$&$	1849265278156800=2^{29}\times 3^{9}\times 5^2\times 7	$\\\hline
$	132\leq k\leq 137	$&$	3698530556313600=2^{30}\times 3^{9}\times 5^2\times 7	$\\\hline
$	138\leq k\leq 147	$&$	22191183337881600=2^{31}\times 3^{10}\times 5^2\times 7	$\\\hline
$	148\leq k\leq 158	$&$	44382366675763200=2^{32}\times 3^{10}\times 5^2\times 7	$\\\hline
$	159\leq k\leq 170	$&$	88764733351526400=2^{33}\times 3^{10}\times 5^2\times 7	$\\\hline
$	171\leq k\leq 181	$&$	2^{34}\times 3^{10}\times 5^2\times 7	$\\\hline
$	182\leq k\leq 194	$&$	2^{35}\times 3^{10}\times 5^2\times 7	$\\\hline
$\hsp	195\leq k\leq 199\hsp	$&$	2^{36}\times 3^{10}\times 5^2\times 7	$\\\hline

\end{tabular}
\end{table}

\vfill\pagebreak

\begin{table}[h]
\caption{\label{table17_2} The values of $N_{\max}(k)$ (Part II).}

\begin{tabular}{|c|l||c|l|}
\hline
Range of $ k  $ & $N_{\max}(k)$ & Range of $  k$ & $N_{\max}(k)$\\\hline
$	 k= 200	$&$	2^{36}\times 3^{11}\times 5^2\times 7	$&$	846\leq k\leq 873	$&$	2^{70}\times 3^{14}\times 5^2	$\\\hline
$	\hsp 201\leq k\leq 213\hsp	$&$ 	2^{37}\times 3^{11}\times 5^2\times 7\hsp	$&$	\hsp 874\leq k\leq 902	\hsp $&$	2^{71}\times 3^{14}\times 5^2	\hsp $\\\hline
$	214\leq k\leq 226	$&$	2^{38}\times 3^{11}\times 5^2\times 7	$&$	903\leq k\leq 930	$&$	2^{72}\times 3^{14}\times 5^2	$\\\hline
$	227\leq k\leq 239	$&$	2^{39}\times 3^{11}\times 5^2\times 7	$&$	931\leq k\leq 943	$&$	2^{73}\times 3^{14}\times 5^2	$\\\hline
$	240\leq k\leq 253	$&$	2^{39}\times 3^{11}\times 5^2	$&$	944\leq k\leq 965	$&$	2^{74}\times 3^{15}\times 5^2	$\\\hline
$	254\leq k\leq 268	$&$	2^{40}\times 3^{11}\times 5^2	$&$	966\leq k\leq 995	$&$	2^{75}\times 3^{15}\times 5^2	$\\\hline
$	269\leq k\leq 284	$&$	2^{41}\times 3^{11}\times 5^2	$&$	996\leq k\leq 1025	$&$	2^{76}\times 3^{15}\times 5^2	$\\\hline
$	285\leq k\leq 299	$&$	2^{42}\times 3^{11}\times 5^2	$&$	1026\leq k\leq 1056	$&$	2^{77}\times 3^{15}\times 5^2	$\\\hline
$	300\leq k\leq 316	$&$	2^{43}\times 3^{11}\times 5^2	$&$	1057\leq k\leq 1087	$&$	2^{78}\times 3^{15}\times 5^2	$\\\hline
$	317\leq k\leq 333	$&$	2^{44}\times 3^{11}\times 5^2	$&$	1088\leq k\leq 1119	$&$	2^{79}\times 3^{15}\times 5^2	$\\\hline
$	334\leq k\leq 343	$&$	2^{45}\times 3^{11}\times 5^2	$&$	1120\leq k\leq 1151	$&$	2^{80}\times 3^{15}\times 5^2	$\\\hline
$	344\leq k\leq 357	$&$	2^{46}\times 3^{12}\times 5^2	$&$	1152\leq k\leq 1184	$&$	2^{81}\times 3^{15}\times 5^2	$\\\hline
$	358\leq k\leq 375	$&$	2^{47}\times 3^{12}\times 5^2	$&$	1185\leq k\leq 1217	$&$	2^{82}\times 3^{15}\times 5^2	$\\\hline
$	376\leq k\leq 393	$&$	2^{48}\times 3^{12}\times 5^2	$&$	1218\leq k\leq 1251	$&$	2^{83}\times 3^{15}\times 5^2	$\\\hline
$	394\leq k\leq 412	$&$	2^{49}\times 3^{12}\times 5^2	$&$	1252\leq k\leq 1264	$&$	2^{84}\times 3^{15}\times 5^2	$\\\hline
$	413\leq k\leq 431	$&$	2^{50}\times 3^{12}\times 5^2	$&$	1265\leq k\leq 1296	$&$	2^{83}\times 3^{15}\times 5	$\\\hline
$	432\leq k\leq 451	$&$	2^{51}\times 3^{12}\times 5^2	$&$	1297\leq k\leq 1332	$&$	2^{84}\times 3^{15}\times 5	$\\\hline
$	452\leq k\leq 471	$&$	2^{52}\times 3^{12}\times 5^2	$&$	1333\leq k\leq 1368	$&$	2^{85}\times 3^{15}\times 5	$\\\hline
$	472\leq k\leq 486	$&$	2^{53}\times 3^{12}\times 5^2	$&$	1369\leq k\leq 1404	$&$	2^{86}\times 3^{15}\times 5	$\\\hline
$	487\leq k\leq 499	$&$	2^{54}\times 3^{13}\times 5^2	$&$	1405\leq k\leq 1441	$&$	2^{87}\times 3^{15}\times 5	$\\\hline
$	500\leq k\leq 520	$&$	2^{55}\times 3^{13}\times 5^2	$&$	1442\leq k\leq 1479	$&$	2^{88}\times 3^{15}\times 5	$\\\hline
$	521\leq k\leq 542	$&$	2^{56}\times 3^{13}\times 5^2	$&$	1480\leq k\leq 1493	$&$	2^{89}\times 3^{15}\times 5	$\\\hline
$	543\leq k\leq 564	$&$	2^{57}\times 3^{13}\times 5^2	$&$	1494\leq k\leq 1518	$&$	2^{90}\times 3^{16}\times 5	$\\\hline
$	565\leq k\leq 587	$&$	2^{58}\times 3^{13}\times 5^2	$&$	1519\leq k\leq 1556	$&$	2^{91}\times 3^{16}\times 5	$\\\hline
$	588\leq k\leq 610	$&$	2^{59}\times 3^{13}\times 5^2	$&$	1557\leq k\leq 1595	$&$	2^{92}\times 3^{16}\times 5	$\\\hline
$	611\leq k\leq 633	$&$	2^{60}\times 3^{13}\times 5^2	$&$	1596\leq k\leq 1634	$&$	2^{93}\times 3^{16}\times 5	$\\\hline
$	634\leq k\leq 657	$&$	2^{61}\times 3^{13}\times 5^2	$&$	1635\leq k\leq 1674	$&$	2^{94}\times 3^{16}\times 5	$\\\hline
$	658\leq k\leq 679	$&$	2^{62}\times 3^{13}\times 5^2	$&$	1675\leq k\leq 1714	$&$	2^{95}\times 3^{16}\times 5	$\\\hline
$	680\leq k\leq 688	$&$	2^{63}\times 3^{14}\times 5^2	$&$	1715\leq k\leq 1755	$&$	2^{96}\times 3^{16}\times 5	$\\\hline
$	689\leq k\leq 713	$&$	2^{64}\times 3^{14}\times 5^2	$&$	1756\leq k\leq 1797	$&$	2^{97}\times 3^{16}\times 5	$\\\hline
$	714\leq k\leq 739	$&$	2^{65}\times 3^{14}\times 5^2	$&$	1798\leq k\leq 1838	$&$	2^{98}\times 3^{16}\times 5	$\\\hline
$	740\leq k\leq 765	$&$	2^{66}\times 3^{14}\times 5^2	$&$	1839\leq k\leq 1881	$&$	2^{99}\times 3^{16}\times 5	$\\\hline
$	766\leq k\leq 791	$&$	2^{67}\times 3^{14}\times 5^2	$&$	1882\leq k\leq 1924	$&$	2^{100}\times 3^{16}\times 5	$\\\hline
$	792\leq k\leq 818	$&$	2^{68}\times 3^{14}\times 5^2	$&$	1925\leq k\leq 1967	$&$	2^{101}\times 3^{16}\times 5	$\\\hline
$	819\leq k\leq 845	$&$	2^{69}\times 3^{14}\times 5^2	$&$	1968\leq k\leq 2000	$&$	2^{102}\times 3^{16}\times 5	$\\\hline

\end{tabular}
\end{table}

\vfill\pagebreak

In Figure \ref{figure1}, we plot the values of $\lambda(k)$ for $2\leq k\leq 2000$.
 
\begin{figure}[ht]
\centering
\includegraphics[scale=0.6]{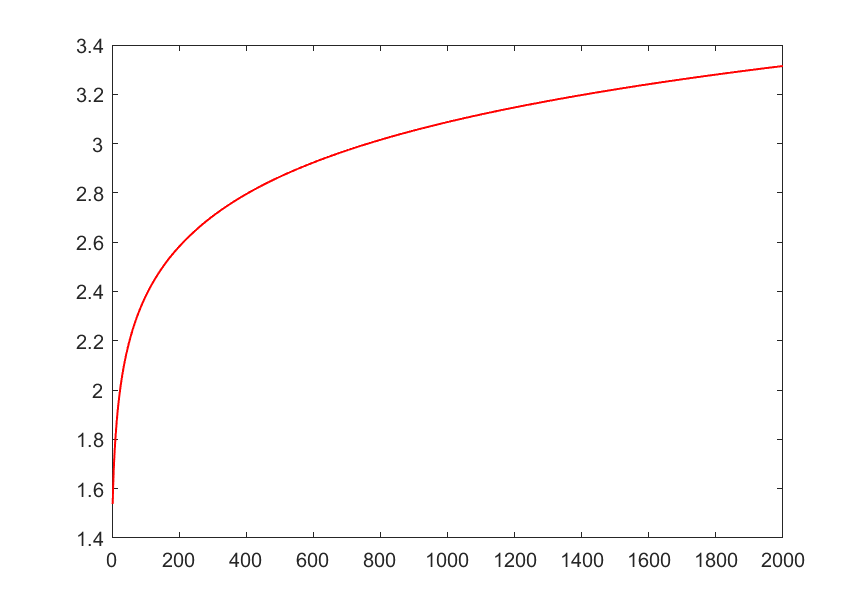}

\caption{The values of $\lambda(k)$ for $2\leq k\leq 2000$.}\label{figure1}
\end{figure}

In this work, we have established an efficient algorithm to solve the optimization problem
of finding the maximum value of the function
\[f_k(n)=\frac{\log d_k(n)\log\log n}{\log k\log n},\quad n\geq 3 \]for any $k\geq 2$. It gives us the smallest positive number $\lambda(k)$ such that
\[\log d_k(n)\leq \frac{\lambda(k)\log k\log n}{\log\log n}\hspace{1cm}\text{for all}\;\;n\geq 3.\]
The concept of superior $k$-highly composite numbers has played an important role. Their properties are worth explored. This will be considered in a future publication.


\bigskip
 
\begin{thebibliography}{10}

\bibitem[Apo76]{Apostol_1}
T.~M. Apostol, \emph{Introduction to analytic number theory}, Undergraduate
  Texts in Mathematics, Springer-Verlag, New York-Heidelberg, 1976.
  \MR{0434929}

\bibitem[BKL21]{Broadbent}
Samuel Broadbent, Habiba Kadiri, Allysa Lumley, Nathan Ng, and Kirsten Wilk,
  \emph{Sharper bounds for the {C}hebyshev function {$\theta(x)$}}, Math. Comp.
  \textbf{90} (2021), no.~331, 2281--2315. \MR{4280302}

\bibitem[DNR99]{Duras}
Jean-Luc Duras, Jean-Louis Nicolas, and Guy Robin, \emph{Grandes valeurs de la
  fonction {$d_k$}}, Number theory in progress, {V}ol. 2
  ({Z}akopane-{K}o\'scielisko, 1997), de Gruyter, Berlin, 1999, pp.~743--770.
  \MR{1689542}

\bibitem[Hep73]{Heppner}
E.~Heppner, \emph{Die maximale {O}rdnung primzahl-unabh\"angiger
  multiplikativer {F}unktionen}, Arch. Math. (Basel) \textbf{24} (1973),
  63--66. \MR{319921}

\bibitem[MV07]{Montgomery}
H.~L. Montgomery and R.~C. Vaughan, \emph{Multiplicative number theory. {I}.
  {C}lassical theory}, Cambridge Studies in Advanced Mathematics, vol.~97,
  Cambridge University Press, Cambridge, 2007. \MR{2378655}

\bibitem[NR83]{Nicolas}
J.-L. Nicolas and G.~Robin, \emph{Majorations explicites pour le nombre de
  diviseurs de {$N$}}, Canad. Math. Bull. \textbf{26} (1983), no.~4, 485--492.

\bibitem[Nic22]{Nicolas_22}
Jean-Louis Nicolas, \emph{Highly composite numbers and the {R}iemann
  hypothesis}, Ramanujan J. \textbf{57} (2022), no.~2, 507--550.  

\bibitem[Ram15]{Ramanujan}
S.~Ramanujan, \emph{Highly composite numbers}, Proc. London Math. Soc. (2)
  \textbf{14} (1915), 347--409.

\bibitem[Ram97]{Ramanujan_2}
Srinivasa Ramanujan, \emph{Highly composite numbers}, Ramanujan J. \textbf{1}
  (1997), no.~2, 119--153, Annotated and with a foreword by Jean-Louis Nicolas
  and Guy Robin.
\bibitem[Rob83]{Robin}
G.~Robin, \emph{M\'ethodes d'optimisation pour un probl\`eme de th\'eorie des
  nombres}, RAIRO Inform. Th\'eor. \textbf{17} (1983), no.~3, 239--247.
  \MR{743888}

\bibitem[RS62]{RosserSchoenfeld1962}
J.~Barkley Rosser and Lowell Schoenfeld, \emph{Approximate formulas for some
  functions of prime numbers}, Illinois J. Math. \textbf{6} (1962), 64--94.

\bibitem[Wig07]{Wigert}
S.~Wigert, \emph{Sur l'ordre de grandeur du nombre des diviseurs d'un
  entier}, Arkiv f\"or matematik, astronomi och fysik \textbf{3} (1907), 1--15.

\end{thebibliography}
\end{document}